\documentclass[11pt]{article}

\usepackage{graphicx}
\usepackage{amsmath,amsthm,amssymb,color,bbm}
\usepackage[noadjust,sort]{cite}
\usepackage[T1]{fontenc}
\usepackage{xcolor}
\usepackage{tikz}
\usetikzlibrary{positioning}

\usepackage[font=small,labelfont=bf]{caption}
\usepackage[normalem]{ulem}

\makeatletter
\g@addto@macro\normalsize{%
  \setlength\abovedisplayskip{8pt plus 3pt minus 3pt}
  \setlength\belowdisplayskip{8pt plus 3pt minus 3pt}
  \setlength\abovedisplayshortskip{6pt plus 3pt minus 2pt}
  \setlength\belowdisplayshortskip{6pt plus 3pt minus 2pt}
}
\makeatother

\date{\today}

\normalsize

\numberwithin{equation}{section}

\def\({\bigl(}
\def\){\bigr)}

\newtheorem{thm}{Theorem}[section]

\newtheorem{lemma}[thm]{Lemma}
\newtheorem{conj}[thm]{Conjecture}
 \newtheorem{claim}[thm]{Claim}

\theoremstyle{definition} 
\newtheorem{remark}[thm]{Remark}

\def\dfrac#1#2{\lower0.15ex\hbox{\large$\textstyle\frac{#1}{#2}$}}
\def\({\bigl(}
\def\){\bigr)}
\def\st{\,:\,}

\def\N{\mathbb{N}}

\let\eps=\varepsilon

\def\b{\boldsymbol{b}}

\def\calS{\mathcal{S}}

\def\1{\mathbbm{1}}

\def\X{\boldsymbol{X}}
\def\Y{\boldsymbol{Y}}

\def\x{\boldsymbol{x}}
\def\y{\boldsymbol{y}}
\def\n{\boldsymbol{n}}
\def\m{\boldsymbol{m}}
\def\u{\boldsymbol{u}}

\def\calF{\mathcal{F}}
\def\calG{\mathcal{G}}

\def\calS{\mathcal{S}}

\def\xvec{\boldsymbol{x}}
\def\yvec{\boldsymbol{y}}
\def\zvec{\boldsymbol{z}}
\def\uvec{\boldsymbol{u}}

\def\l{\boldsymbol{\ell}}

\def\mad{\operatorname{mad}}

\def\E{\operatorname{\mathbb{E}}}

\def\Pr{\operatorname{Pr}}

\def\Reals{{\mathbb{R}}}
\def\pReals{\Reals_+}

\def\Naturals{{\mathbb{N}}}

\def\nb{k} 
\def\I{\mathcal{I}} 

\def \M{\mathcal{G}} 
\def \SBM{\mathcal{G}(\n,P)} 

\def \G{\boldsymbol{G}} 

\begin{document}

\title{On the chromatic number  in  the stochastic block model}
\author{
Mikhail Isaev\thanks{Supported by Australian Research Council Discovery Project DP190100977
and by Australian Research Council Discovery Early Career Researcher Award DE200101045}\\
\small School of Mathematics\\[-0.8ex]
\small Monash University\\[-0.8ex]
\small 3800  Clayton,   Australia\\[-0.8ex]
\small\tt  mikhail.isaev@monash.edu    
\and
Mihyun~Kang\thanks{Supported by Austrian Science Fund (FWF): I3747}\\
\small Institute of Discrete Mathematics\\[-0.8ex]
\small Graz University of Technology\\[-0.8ex]
\small  8010  Graz, Austria\\[-0.8ex]
\small\tt  kang@math.tugraz.at
 }
\date{}


\maketitle

\begin{abstract} 
We prove a generalisation of Bollob\'as' classical result 
on the asymptotics of the chromatic number of the binomial random graph 
to the stochastic block model. In addition, by allowing  the number of blocks  to grow,
we determine  the chromatic number in the Chung-Lu model. Our approach is based on the estimates for the  weighted independence number, where weights are specifically designed  to  encapsulate inhomogeneities of the random graph.
 \end{abstract}

\section{Introduction}\label{S:intro}
 
The chromatic number  $\chi(G)$ of a graph $G$, denoted by $\chi(G)$,  is  the smallest number of colours needed for the assignment of colours to the vertices of $G$ so that no two adjacent vertices have the same colour. Understanding properties of the distribution of $\chi(\G)$ for random $\G$  is  one of the most prominent problems in the random graph theory since the seminal paper~\cite{ER60} by Erd{\H o}s and R\'enyi.

The binomial random graph $\G(n,p)$ is the most studied in the literature.  
Recall that  $\G(n,p)$ is  a graph on vertex set $[n]:=\{1,2,\ldots,n\}$  and each pair of distinct vertices is connected by an edge independently of 
  each other with probability $p$. 
A long line of research  led    
to  many breakthrough results on the asymptotic behaviour and concentration of 
 $\chi(\G(n,p))$; see~\cite{AchlioptasNaor2005, AlonKrivelevich1997, Bollobas1988, Bollobas2004, CPA2008, Heckel2018, Heckel2020, KrivelevichSudakov1998, Luczak1991, Luczak1991b, McDiarmid1990, Panagiotou2009, ShamirSpencer1987} --- this list is far from being exhaustive. In particular, it is well known that if   $p=p(n) \in [0,1]$ is such that $np\rightarrow \infty$ as $n \rightarrow \infty$  and $p\leq 1-\eps$ for some fixed $\eps>0$, then, whp (meaning {\em with probability tending to one}) as $n \rightarrow \infty$,
\begin{equation}\label{eq:Gnp}
	\chi(\G(n,p)) = (1+o(1))\frac{ n \log\left( \frac{1}{1-p} \right)}{ 2 \log (pn) }.
\end{equation}
Formally,   ``$\X(n) = (1+o(1)) \Y(n)$  holds whp as $n \rightarrow \infty$''   means that,  for any fixed
 $\epsilon>0$, the probability of the event that 
$ (1-\epsilon)\Y(n)\leq \X(n) \leq (1+\epsilon)\Y(n)$  tends to 1 as $n \rightarrow \infty$.
Throughout the paper we use  $\log$ to denote the natural logarithm.

Our  paper focuses on generalising formula \eqref{eq:Gnp} to a random graph
$\G$ from  \emph{the stochastic block model}, in which  all vertices are distributed between several different blocks and the probabilities of adjacencies of vertices depend only on the block they belong to; see Section~\ref{S:SBM} for  formal definitions.   
The chromatic number in this random graph model was recently studied by 
  Martinsson et al. \cite{MPSM2020}.  
  Under the condition that the number of blocks is fixed and all probabilities are constants from 
  $(0,1)$, namely, they are all independent of  the number of vertices $n$, Martinsson et al.  proved that,  whp as $n\rightarrow \infty$,
  \[
  	\chi(\G) = (1+o(1)) \frac{n}{ c^* \log n},
  \]
  where constant $c^*$ is  the solution of a certain convex optimisation problem,
which   depends only on the matrix of probabilities and the proportions for the distribution of  $n$  vertices between the blocks. 

 In this paper we  extend the above result by Martinsson et al. \cite{MPSM2020} in two directions:  
 \begin{enumerate}
 \item[(1)]  
the edge probabilities  can be functions of $n$ (in particular, vanishing or tending to $1$),    
\item[(2)]  the number of blocks  can grow as a function of $n$. 
\end{enumerate}
We defer  the exact statement of our main result (Theorem~\ref{Thm_sbm}) to Section~\ref{S:SBM} in order to obviate introducing the technical notations in the  introduction. In the rest of this section we discuss several consequences of Theorem~\ref{Thm_sbm}, which are interesting of its own.

\subsection{A very dense binomial random graph}\label{S:verydense}
The classical binomial random graph $\G(n,p)$ can be considered as a random graph from the stochasitic block model with a single block. Even in this case,   our main result (Theorem~\ref{Thm_sbm}) implies   new information on the chromatic number of a very dense binomial random graph when $p= p(n) \rightarrow 1$ as $n \rightarrow \infty$ which was not treated in the literature. Namely, as a straightforward application of Theorem~\ref{Thm_sbm}, we obtain the following result.
\begin{thm}\label{T:oneblock}
 If   {$p= p(n) \in [0,1]$ such that  $p \rightarrow 1$} and  $1-p = n^{o(1)}$, 
 then  \eqref{eq:Gnp} holds whp.
 \end{thm}
 
We believe that $n^{o(1)}$ in Theorem \ref{T:oneblock} can not be improved. For example, if 
$p_1= 1 - \frac{1}{n \log n}$ then whp $\G(n, p_1)$ has a clique of size $(1+o(1)) n$ since its  complement   contains $o(n)$ edges. Thus,  whp as $n \rightarrow \infty$
\[
	\chi(\G(n,p_1))  = (1+o(1)) n.
\]
 On the other hand, if $p_2 = 1 - \frac{\log^2 n}{n}$   
	then whp the complement of   $\G(n, p_2)$ contains a perfect matching as shown by  Erd\H os and R\'enyi \cite{ER66}. 
	Thus,  whp  as $n \rightarrow \infty$ 
	\[
		\chi(\G(n,p_2))  \leq (1+o(1)) \frac{n}{2}.
	\]
Note  that formula   \eqref{eq:Gnp}  is not valid for  $p= p_1$, but it might still be true for $p=p_2$, because 
\[
	\frac{  \log \frac {1}{1-p_1}}{ \log ( p_1 n)}  = 1+o(1) \qquad  \text{and} \qquad \frac{  \log \frac {1}{1-p_2}}{ \log ( p_2 n)}  = 1+o(1).
\]

More generally, for the case when $p = 1- n^{O(1)}$, we conjecture the following.
\begin{conj}\label{Con:Gnp} 
 Let  $r\ge 2$ be a fixed integer and $p = p(n)\in [0,1]$ be such that  
 \[
 n^{-\frac{2}{r+1}} \gg 1-p \gg n^{-\frac{2}{r}}.
 \] 
 Then,  $\chi(\G(n,p)) = (1+o(1)) \frac{n}{r}$  whp as $n \rightarrow \infty$ . 
 \end{conj}
 
As observed above, 
$\G(n,p)$ can be coloured  in $\frac{n}{2}$ colours if its complement has a perfect matching.
In fact, to achieve a colouring  with at most 
$(1+o(1)) \frac{n}{2}$ colours,  it is sufficient that the complement  of $\G(n,p)$ contains an almost perfect matching covering $n-o(n)$ vertices.   
Similarly, for any fixed integer $r\ge 2$, in order  to show that $\chi(\G(n,p)) \leq (1+o(1)) \frac{n}{r}$, it is sufficient to find an almost perfect $K_{r}$-matching in the complement of $\G(n,p)$. (Throughout the paper, $K_r$ denotes the complete graph with vertex set $[r]$ or the clique of size $r$.) For an arbitrary graph $G$, the  thresholds  for the existence of perfect $G$-matchings and almost perfect $G$-matchings was studied by Ruci\'nski \cite{Rucinski1992}
and  by  Johansson, Kahn, and Vu \cite{JKV2008}. In particular, \cite[Theorem 4]{Rucinski1992} establishes the existence of an almost  perfect $K_{r}$-matching if 
$n (1-p)^{r/2} \gg 1$ which implies the upper bound of Conjecture \ref{Con:Gnp}. However, the lower bound for $\chi(\G(n,p))$ does not follow from the known results  on $G$-matchings since an optimal colouring migh have  colour classes of different sizes.

{Conjecture \ref{Con:Gnp} was recently confirmed by Surya and Warnke; see  \cite[Theorem 13]{SW2022}.}

 \subsection{Percolations on blow-up graphs}

 Given a graph $G=(V(G),E(G))$ and $p \in  (0, 1)$,  the percolated random graph $G_p$, which is also known as a random subgraph of $G$, is generated from $G$ by keeping each edge in $E(G)$ independently with probability $p$. In particular, if $G=K_n$, then $G_p$ is equivalent to  the binomial random graph  $\G(n,p)$. 
In this case, formula
	 \eqref{eq:Gnp} can read as follows: whp 
	 \begin{equation}\label{eq:perc}	
	 	\chi(G_p) = (1+o(1)) \frac{ \log (\frac {1}{1-p})}{2 \log ( p n)}\, \chi(G), \qquad 
	 	\text{as $n  =  |V(G)| \rightarrow \infty$.}
 	 \end{equation}

 In this paper we show that \eqref{eq:perc} holds   when  $G$ is  a  \emph{blow-up graph} $G_H(\n)$
  constructed as follows. Given a graph $H$ on vertex set $[k]$ and 
a vector $\n =  (n_1,\ldots,n_k)^T \in \Naturals^k$,   we denote by $G_H(\n)$ the graph  obtained from $H$ by replacing  each vertex $i \in [k]$ 
with $K_{n_i}$. An edge between any two vertices from different cliques appears in $G_H(\n)$ if the corresponding edge is present in $H$.
One  can consider  the blow-up graph $G_H(\n)$ as a special case of a ``random'' graph  from 
the  stochastic block model     by setting all probabilities $1$ or $0$ according to the adjacency matrix of the graph $H$.

\begin{figure}[h!]
	\centering
	\begin{tikzpicture}[scale=0.6]
		
		\draw [ultra thick, rounded corners] (-6,6) rectangle (-4,8);
			\node (u1) at (-5,7) {$K_{n_1}$};

		\draw [ultra thick, rounded corners] (-2,6) rectangle (0,8);
		\node (u2) at (-1,7) {$K_{n_2}$};

			\draw [ultra thick, rounded corners] (2,6) rectangle (4,8);
		\node (u3) at (3,7) {$K_{n_3}$};

			\draw [ultra thick, rounded corners] (-2,2) rectangle (0,4);
		\node (u4) at (-1,3) {$K_{n_4}$};

		\draw [ultra thick, rounded corners] (2,2) rectangle (4,4);
		\node (u5) at (3,3) {$K_{n_5}$};
		
	\draw[thick] (u1) edge[bend left=16] (u2);
			\draw[thick] (u1) edge[bend left=-32] (u2);
			\draw[thick] (u1) edge[bend left=32] (u2);
		\draw[thick] (u1) edge  (u2);
		\draw[thick] (u1) edge[bend right=16] (u2);

	\draw[thick] (u2) edge[bend left=16] (u3);
\draw[thick] (u2) edge[bend left=-16] (u3);
\draw[thick] (u2) edge[bend left=0] (u3);
		\draw[thick] (u3) edge[bend left=-32] (u2);
		\draw[thick] (u3) edge[bend left=32] (u2);
	
	\draw[thick] (u2) edge[bend left=16] (u4);
	\draw[thick] (u2) edge[bend left=-16] (u4);
	\draw[thick] (u2) edge[bend left=0] (u4);	
			\draw[thick] (u4) edge[bend left=-32] (u2);
		\draw[thick] (u4) edge[bend left=32] (u2);

		\draw[thick] (u5) edge[bend left=16] (u4);
		\draw[thick] (u5) edge[bend left=-16] (u4);
		\draw[thick] (u5) edge[bend left=0] (u4);	
		\draw[thick] (u4) edge[bend left=-32] (u5);
		\draw[thick] (u4) edge[bend left=32] (u5);

			\draw[thick] (u5) edge[bend left=16] (u3);
		\draw[thick] (u5) edge[bend left=-16] (u3);
		\draw[thick] (u5) edge[bend left=0] (u3);	
		\draw[thick] (u3) edge[bend left=-32] (u5);
		\draw[thick] (u3) edge[bend left=32] (u5);
		
		\begin{scope}[shift={(12,0)}]
			
				\node (v1) [label=above:{$1$}] at (-5,7) {};
				\draw [fill] (v1) circle (0.2);
			
				\node (v2) [label=above:{$2$}] at (-1,7) {};
			\draw [fill] (v2) circle (0.2);
			
			 	\node (v3) [label=above:{$3$}] at (3,7) {};
			 \draw [fill] (v3) circle (0.2);

			 	\node (v4) [label=left:{$4$}] at (-1,3) {};
			 \draw [fill] (v4) circle (0.2);

			 	\node (v5) [label=right:{$5$}] at (3,3) {};
			 \draw [fill] (v5) circle (0.2);
 
 	\draw[thick] (v1)--(v2)--(v4)--(v5) -- (v3) -- (v2);
 
		\end{scope}
	\end{tikzpicture}
	\caption{A blow-up graph $G_H(\n)$ (left) for a graph $H$ on 5 vertices (right).}
	\label{f:blow-up}
\end{figure}
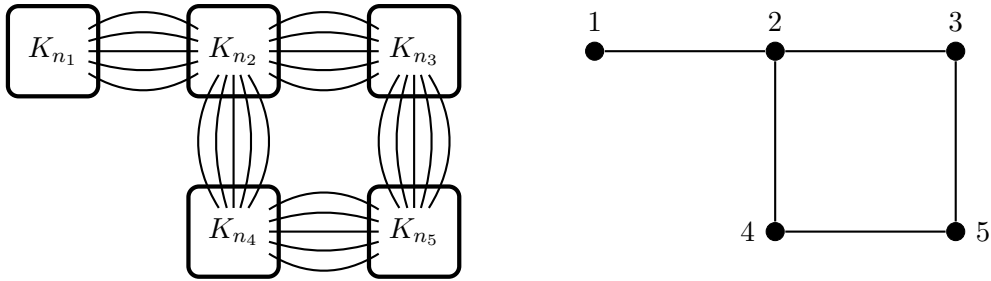

Everywhere in this paper  the norm notation $\|\cdot\|$ stands for the $1$-norm:
\[
\|\n\|= n_1 +\cdots +n_k.
\] 
\begin{thm}\label{T:perc} 
Let $\epsilon \in (0,\frac14)$ be fixed and $H$ be a graph with vertex set $[k]$. Assume $\n = \n(n) \in \Naturals^k$  and $p=p(n) \in (0,1)$ are such that  as $n \rightarrow \infty$, 
 \[
 	{ \|\n\| \rightarrow \infty,\qquad 
 	  p \geq  \|\n\|^{-\frac{1}{4} +\epsilon}, \qquad 
  1-p   =    \|\n\|^{-o(1)}.}
 \]
   Then,    
	   \eqref{eq:perc} with $G = G_H(\n)$  holds whp.
\end{thm} 
We prove Theorem \ref{T:perc} in Section \ref{S:perc}.  
Note that Theorem \ref{T:perc} with  $k=1$ and $n_1=n$ (and thus $G_H(\n) = K_n$) 
 recovers Theorem~\ref{T:oneblock}.

Determining the chromatic number of a random subgraph $G_p$ for a general graph $G$
is a much harder problem; see, for example, ~\cite{AlonKrivelevichSudakov1997, Berkowitz,Bukh,MoharWu,Shinkar}. In particular, Bukh asks {\cite{Bukh}}  whether for any graph $G$, there exists a positive constant $c$ such that  
$\mathbb E \chi(G_{{1}/{2}}) \ge \frac{c}{\log (\chi(G))}\, \chi(G).$ 
Using standard concentration results,   Bukh's question  {for blow-up graphs} 
  is equivalent to that  whp
 \begin{equation*}
  \chi(G_{{1}/{2}}) \ge \frac{c}{\log ( |V(G)|)}\, \chi(G).
  \end{equation*}
  Theorem \ref{T:perc} establishes this bound for blow-up graphs.  
  It would be interesting to find other classes of graphs that satisfy  \eqref{eq:perc} (or at least its lower bound).



 \subsection{Chung-Lu model}
 As mentioned, our main result (Theorem~\ref{Thm_sbm}) allows the number of blocks to grow. Thus,   
 one can study $\chi(\G)$ for  general inhomogenous random graphs $\G$ using approximations by the stochastic block model.     To demonstrate the idea, we consider the following two random graph models. 
  Given $\uvec  = (u_1,\ldots, u_n)^T \in [0,1]^n$ and $p \in [0,1]$, a  random graph 
 $\G^{\times}_p \sim \M^{\times}(\uvec,p)$  has vertex set $[n]$ and edges $ij$  are generated independently of each other with probabilities
\[
	p_{ij}^{\times} = p\, u_i u_j \quad\quad   i,j \in [n].
\]
Similarly, given $\uvec \in [0,1]^n$ and $p \in [0,\dfrac12]$,  a  random graph 
 $\G^{+}_p \sim \M^{+}(\uvec,p)$  has vertex set $[n]$ and  edges  $ij$ are generated independently of each other with probabilities
\[
	p_{ij}^{+} = p\, (u_i + u_j)\quad\quad   i,j \in [n].
\]

The model $ \M^{\times}(\uvec,p)$ is known as  the Chung-Lu  random graph model  and it is of central importance in the network analysis; for more extensive background, see, for example,  \cite{CL2002} and references therein. For decreasing $p = p(n)$, the model $ \M^{+}(\uvec,p)$ is asymptotically equivalent to the   complement of  the Chung-Lu    model.

 \begin{thm}\label{T:main4}
		Let  $\epsilon>0$ {be fixed} and  $p = p(n)$ be such  $1 \gg p \geq n^{-\frac14 +\epsilon}$
		as $n \rightarrow \infty$.
		Then,   whp  uniformly over $\uvec\in [0,1]^n$ satisfying 
		$\sum_{i\in[n]}u_i   = \Omega(n)$,  the following hold:
\begin{itemize}		
		\item[(a)] $\displaystyle 	\chi(\G_p^{\times})   =  (1+o(1))\frac{ p }{2 \log (pn)}\, 
		\max_{U \subseteq [n]} \frac{1}{|U|} \left(\sum_{i\in U} u_i\right)^2$,
	 where $\G_p^{\times} \sim \M^{\times}(\uvec,p)$;
	 \item[(b)]
	 $
	 	\displaystyle
	 	\chi(\G_p^{+}) =  (1+o(1))\frac{ p }{\log (pn)} \, \sum_{i\in [n]} u_i,
	 	$
	  where $\G_p^{+} \sim \M^{+}(\uvec,p).$ 
		\end{itemize}
%
\end{thm}

We prove Theorem \ref{T:main4} in Section \ref{S:main4}.  
Theorem \ref{T:main4} applies to the case  when {a constant fraction} of  expected degrees of the random graphs  
$\G_p^{\times}$
and $\G_p^{+}$ are within a multiplicative constant of {the maximum expected degree}.    We believe that  the formulas of Theorem \ref{T:main4}  can be extended to allow a larger variation of components of $\uvec$ covering, for example, power-law degree sequences.


\section{Stochastic block model}\label{S:SBM}

Before stating our main result on the chromatic number of a random graph from the stochastic block model,  we first define the stochastic block model formally.
For a positive integer $\nb$, a  vector $\n=(n_1,\ldots,n_{\nb})^T\in \mathbb{N}^{\nb}$, and 
a $\nb \times \nb$ symmetric matrix  $P=(p_{ij})_{i,j \in [k]}$ with $p_{ij}  \in [0,1]$,  
a  random graph $\G$ from the stochastic block model  $\SBM$,  denoted by $\G \sim   \SBM$,   is constructed as follows: 
\begin{itemize}
	\item the vertex set $V(\G)$ is partitioned into $k$ disjoint blocks 
	$B_1,\ldots, B_k$  of  sizes $|B_i| = n_i$ for $i \in [k]$ (and we write $V(\G) = B_1 \cup \cdots \cup B_k$);   
	\item  each pair $\{u, v\}$ of distinct vertices $u,v \in V(\G)$ is included in the edge set $E(\G)$,  
	independently of one another,  with probability 
	$$ 
	p(u,v) := p_{ij},
	$$
	where $i = i(u) \in [k]$ and $j = j(v) \in [k]$ are such $u \in B_i$ and $v \in B_j$.
\end{itemize}

{Throughout the paper, for all asymptotic notation, 
we implicitly consider sequences of vectors $\n= \n(n) \in \Naturals^k$ and  $k\times k$ symmetric matrices $P=P(n)$, where
\[
k = k(n), \qquad\n(n) = (n_1(n), \ldots n_k(n))^T, \qquad  P = \Big(p_{ij}(n)\Big)_{i,j\in [k]}.
\]
 Our    bounds (including whp results) hold uniformly over all   sequences  $\n(n)$ and  $P(n)$, where $n \rightarrow \infty$,   satisfying stated  assumptions where the implicit functions like in $o(\cdot)$ depend on $n$ only.  
 Apart from the standard Landau notation $o(\cdot)$ and $O(\cdot)$,
 we also  use the notation  $a_n=\omega(b_n)$ or $a_n=\Omega(b_n)$  if $a_n>0$ always
and $b_n=o(a_n)$ or $b_n=O(a_n)$, respectively.
We write $a_n = \Theta(b_n)$ if $a_n = O(b_n)$ and $b_n = O(a_n)$. 
If both $a_n$ and $b_n$ are positive sequences, we also write
$a_n\ll b_n$ if $a_n=o(b_n)$, and $a_n\gg b_n$  if $a_n=\omega(b_n)$. 
For example, $k= \|\n\|^{o(1)}$ means that $\frac{\log k(n)}{\log \|\n(n)\|} \rightarrow 0 $.
}

In the following, we always assume that  $p_{ij}=p_{ji}$ and $0\le p_{ij}<1$ for all $i,j \in [k]$.
Define the $k\times k$ symmetric matrix  $Q=Q(P)$ by 
\begin{equation}\label{Q_def}
Q:= (q_{ij})_{i,j \in [k]}, \qquad \text{where}\ \ q_{ij} :=   \log \left(\dfrac{1}{1-p_{ij}}\right).
\end{equation}
Let  $\pReals :=[0,+\infty)$ and, for $\xvec, \yvec \in \Reals^k$,   we denote 
\[ 
\yvec \preceq \xvec  \text{ whenever{$\xvec-\yvec \in \pReals^k$}.}
\]
Let  $w(\cdot,Q): \pReals^k  \to  \pReals$ be defined by 
\begin{equation}\label{w_def}
w(\xvec,Q)  := \max_{\boldsymbol{0} \preceq \yvec \preceq \xvec}\ \frac{\yvec^T\, Q\, \yvec}{ \|\yvec\|},
\qquad{\xvec \in \pReals^k}, 
\end{equation}
where  $\boldsymbol{0} = (0,\ldots,0)^T \in \Reals^k$ and
$
\|\yvec\| := |y_1| + \ldots + |y_k|.
$
In \eqref{w_def},  we take $\frac{\yvec^T\, Q\, \yvec}{ \|\yvec\|}$   to be zero for $\yvec = \boldsymbol{0}$, so it is   a continuous function of $\yvec$, which  achieves its 
the maximal value on the compact set   $\{\yvec \in \pReals^k:   \yvec \preceq \xvec\}$. 
In fact, it is always achieved  at a corner, where    $y_i \in \{0,x_i\}$ for all $i\in [k]$; see{Theorem \ref{l:Qmatrix}(b). } 

{The quantity $w(\xvec,Q)$ is closely related to the minimum number of colours required to properly colour  an inhomegeneous graph  with "balanced" colour classes. 
	To illustrate it, let us consider a random graph $\G \sim \SBM$, where $\n = (nx_1, n x_2, \ldots, nx_k)^T = n \xvec$  and number of blocks $k$ 
	and all probabilities $p_{ij}\in (0,1)$  are fixed. In order to determine the size of largest "balanced" independent set, we will present here some rough first moment calculations, while the full details are given in Section~\ref{S:weighted}
	and  Section~\ref{S:optimal}.}

\begin{figure}[h!]
	\centering
	\begin{tikzpicture}[scale=0.6]

		\draw [ultra thick, rounded corners] (-10,1) rectangle (-8,5);
		\node at (-9,4.2) {$nx_1$};
		
			\node at (-9,1.35) {$sx_1$};

		\draw [ultra thick, rounded corners] (-7,0) rectangle (-5,8);
		\node  at (-6,7.2) {$nx_2$};;
		
			\node  at (-6,1.1) {$sx_2$};;
		
			\draw [ultra thick, rounded corners] (-4,0) rectangle (-1,10);
		\node   at (-2.5,9.2) {$nx_3$};
		
			\node   at (-2.5,1.1) {$sx_3$};
		
		\node at (1,4.5) {\LARGE $\mathbb{\cdots}$};
		
	 	\draw [ultra thick, rounded corners] (5,1) rectangle (3,5);
	 \node   at (4,4.2) {$nx_k$};
	 
	  \node   at (4,1.35) {$sx_k$};
	  
	  \node at (-10.2, -0.2) {$\boldsymbol{S}$};
	 
	 \draw [ultra thick, dashed, rounded corners]  (-2.5,1) ellipse (8 and 1.5);
	 

	\end{tikzpicture}
	\caption{A  "balanced"  set  $S = \bigcup_{i\in [k]}S_i$ in  $\G \sim \SBM$, where 
		$|S_i| = sx_i$ and $n_i = n x_i$.}
	\label{f:balanced}
\end{figure}
{ 
The expected number of  collections of $k$ disjoint sets $S_i$, each of which takes $s x_i$ vertices from each block $B_i$, such that $\cup_{i\in [k]}S_i$ is an independent set in $\G$ (see Figure  \ref{f:balanced}) is given by \begin{equation}\label{w-intuition}
	\prod_{i \in [k]} \frac{\binom{nx_i}{sx_i}}{(1-p_{ii})^{sx_i}} \cdot \prod_{i,j\in [k] } (1-p_{ij})^{s^2 x_i x_j} 
	 =  \exp\left(-\dfrac{s^2}{2}\xvec^T Q \xvec + O(s \|\xvec\|)\right) \left(\frac{en}{s}\right)^{s \|\xvec\|},
\end{equation} 
where the RHS is derived  via Stirling's formula for any slowly growing $s = s(n) \ll \sqrt{n}$. 
The first moment threshold corresponds to 
\[
\exp\left(-\dfrac{1}{2}s\xvec^T Q \xvec\right) \left(\frac{en}{s}\right)^{\|\xvec\|} =1,
\]
which gives
\[
s  \approx 2 \log n \cdot \frac{\|\xvec\|}{\xvec^T Q\xvec}.
\] However, this might be significantly  above the existence threshold due to the fact that 
our random graph model is inhomogeneous. In particular,  
the appearance of "balanced" independent sets in $\G$ implies the existence of a "balanced" independent set (with the same size proportion $s/n$)  in its subgraph $\G' \sim \calG(\n',P)$ where
$\n'  = n \yvec$ for any $\boldsymbol{0} \preceq \yvec  \preceq \xvec $.
Repeating the arguments of \eqref{w-intuition} for such $\G'$, we conclude that whp
$s$ can not exceed  
\[ 
{2\log n} \cdot  \min_{\boldsymbol{0} \preceq \yvec \preceq \xvec}\ \frac{ \|\yvec\|} {\yvec^T\, Q\, \yvec}= 2\log n \cdot \frac{1}{ w(\xvec,Q)}.
\] 
In Section \ref{S:weighted}, we  show  that it is indeed the existence threshold (for a more general setting that allows vanishing probabilities); see Theorem \ref{T:independent}.
}

Define  $w_*(\cdot,Q): \pReals^k\, \to \ \pReals$  by
\begin{equation}\label{def:mu-star}
w_*(\xvec,Q) := 
\inf_{\calS \in \mathcal{F}(\xvec)} \  \sum_{ \yvec \in \calS}\ w(\yvec,Q),       \qquad{\xvec \in \pReals^k},
\end{equation}
where 
$ \mathcal{F}(\xvec)$ 
consists of finite systems  $\calS$ of vectors from $\pReals^k$ such that $ \sum_{\yvec \in \calS}  \yvec = \xvec.$
In fact, the infimum of  $  \sum_{ \yvec \in \calS} w(\yvec)$ in \eqref{def:mu-star} is  always achieved by a system $\calS \in  \mathcal{F}(\xvec)$ consisting of  at most $k$ vectors; see Theorem \ref{l:Qmatrix}(f).    

{Similarly to $w(\xvec,Q)$, the quantity $w_*(\xvec,Q)$ has a combinatorial meaning as follows. Consider again a random graph $\G \sim \SBM$, where $\n = (nx_1, n x_2, \ldots, nx_n)^T  = n \xvec$ 
	and all probabilities $p_{ij}\in (0,1)$  are fixed. Then whp the {\em minimum} number of colours required to properly colour $\G$ utilising 
	at most $k$ different types of independent sets 
is	asymptotically equal to   
\[   
\frac{n}{2\log n}w_*(\xvec,Q).
\]
The next theorem shows that such colourings are asymptotically optimal, that is, no more than $k$ different types are required to determine $\chi(\G)$  (for a more general setting that allows vanishing probabilities). 
Let 
\begin{equation}\label{def:q-star}
	q^*:= \max_{i\in [k]}q_{ii} \qquad \text{and} \qquad \hat q(\xvec):= \frac{\sum_{i\in [k]} x_i q_{ii}}
	{\|\xvec\|},  \quad {\xvec\neq \boldsymbol{0}.}
\end{equation}
For convenience, we also set $\hat{q}(\boldsymbol{0} ):= q^*$.
}

\begin{thm}\label{Thm_sbm}
	Let $\sigma \in [0, \sigma_0]$  for some fixed  $0<\sigma_0 < \frac14$
	 and 
	let   $P = (p_{ij})_{i,j\in [k]}$ be such that  $p_{ij}=p_{ji}$ and $0\le p_{ij}<1$ for all $i,j \in [k]$.
	Let {$Q:= (q_{ij})_{i,j \in [k]}$ where  $q_{ij} :=   \log \left(\dfrac{1}{1-p_{ij}}\right)$.  Let $q^*$,   $\hat{q}(\cdot)$, and $w_*(\cdot, Q)$
	be  defined  by  \eqref{def:q-star} and  \eqref{def:mu-star}.}
	Assume that  the following asymptotics hold:
	\begin{equation}\label{ass_sbm}
	 \|\n\| \rightarrow \infty, \qquad  k  = \|\n\|^{o(1)}, \qquad  q^* =    \|\n\|^{-\sigma+o(1)}, \qquad \hat{q}(\n) =    \|\n\|^{-\sigma+o(1)}. 
	\end{equation}
	Assume also that  
	\begin{equation}\label{ass_1}
	\left(1+ \dfrac{1}{q^*}\right)\max_{i,j \in [k]} q_{ij} \ll \log\|\n\| 
	\end{equation}
	 and 
	\begin{equation}\label{ass_2}
	w_*(\n,Q) \gg  k \hat q(\n)  q^*\frac{ \|\n\|    }{\log \|\n\|}.
	\end{equation}
	Then, whp 
	\begin{equation*}\label{eq_sbm}
	\chi(\G) = 
	(1+o(1))\frac{w_*(\n,Q)}{ 2(1-\sigma)\log \|\n\| }, \quad \text{ where  $\G \sim \SBM$.}
	\end{equation*}
\end{thm}

\smallskip
 \begin{remark}
	{In Theorem~\ref{Thm_sbm}, the parameter $\sigma \in [0,\sigma_0]$ governs the density of $\G$. It is convenient for our examples to have  it  not fixed, but treat $\sigma$  as a bounded parameter appearing in the formula for the chromatic number.}
	We believe that the condition $
	\sigma_0< \frac14$    	is an artefact of our proof techniques. Similarly to  the dense case in \cite[Section 7.4]{FK2015} and also to \cite{MPSM2020}, we rely on  Janson's inequality to find sufficiently large independent sets inside any subset of remaining vertices.  Generalisations of the techniques used by
	\L uczak \cite{Luczak1991}  should   extend  Theorem \ref{Thm_sbm} to any $\sigma_0 <1$.
\end{remark}

\begin{remark}
{	Informally, the assumptions of  \eqref{ass_sbm}	 say that  the number of blocks in $\SBM$ is not too big (sublinear in $ \|\n\|$) and the maximum edge probability within a block deviates  not too much (also by a sublinear in $\|\n\|$ factor)   from the average probability within blocks.  Next, the behaviour of  edge probabilities between blocks  is limited by  assumption \eqref{ass_1}: they can vary much more significantly than the diagonal probabilities, but we prohibit them to converge to $1$ too quickly. Note that we allow some edge probabilities  to be small and even $0$, but the upper bounds on the maximal probabilities are essential as demonstrated in Section \ref{S:verydense}. Finally,  \eqref{ass_2} is a technical assumption that is usually not very hard to verify.  In particular, it follows from a stronger but more explicit condition
	$(k q^*)^2 \ll \hat q(\n) \log \|\n\|$; see the lower bound of Theorem~\ref{l:Qmatrix}(d).}
	\end{remark}

\subsection{Proof of Theorem \ref{Thm_sbm}}\label{S:proof_sbm}

In this section, we provide the proof  of  Theorem \ref{Thm_sbm}  based  on  two explicit probability estimates   for $\chi(\G)$ to satisfy the upper and the lower bound stated below.

\begin{thm}\label{T:lower}
	Let $\G \sim \SBM$,   where $P= (p_{ij})_{i,j\in [k]}$
	is such that $p_{ij}=p_{ji}$ and $0\le p_{ij}<1$ for all $i,j\in [k]$.	 Assume $\|\n\|\rightarrow \infty$ and   
	\eqref{ass_1} holds.
	Then,  for any $\eps>0$,
	\begin{equation}\label{eq_lower}
	\Pr\left(\chi(\G) < (1-\eps)\frac{ w_*(\n)}{2 \log  (q^*\|\n\|)) }\right)
	\leq  \exp\left( - \Omega\left( \frac{ \log(q^*\|\n\|)}{\max_{i,j\in [k]}q_{ij}} \right)\right).
	\end{equation}
\end{thm}
We  prove Theorem \ref{T:lower}  in Section~\ref{S:lower}.  
This  lower tail bound  follows from the existence of large weighted independent sets, similarly to arguments of Bollob\'as~\cite{Bollobas1988} and {\L}uczak~\cite{Luczak1991}.
 Comparing to the assumptions of  Theorem \ref{Thm_sbm}, we 
note that Theorem \ref{T:lower}  also applies for sparser graphs with $q^* < \|\n\|^{-\frac14}$,{because it does not require assumption  \eqref{ass_sbm} to hold.}

\begin{thm}\label{T:upper}
	Let $\G \sim \SBM$,   where $P = (p_{ij})_{i,j\in [k]}$ is such that $p_{ij}=p_{ji}$ and $0\le p_{ij}<1$  for all $i,j \in [k]$.   	Let $\sigma \in [0, \sigma_0]$ for some fixed $0<\sigma_0<\frac14$.  
	 Assume    that \eqref{ass_sbm} and \eqref{ass_2} hold.  
	Then,    for any $\eps>0$,
	\begin{equation}\label{eq_upper}
	\Pr\left( \chi(\G) >  \left(1 + \eps \right) \frac{w_*(\n)}{2 \log (q^*\|\n\|)} \right) \leq \exp \left(- \|\n\|^{2- 4\sigma  +o(1)}\right).
	\end{equation}
\end{thm}    
We  prove Theorem \ref{T:upper}  in Section~\ref{S:upper}, using/extending some results and standard arguments on the chromatic numer of the classical binomial random graph.  
Comparing to the assumptions of  Theorem \ref{Thm_sbm}, we 
note that Theorem \ref{T:upper}    allows more variation in the off-diagonal probabilities $p_{ij}$,{because it does not require assumption  \eqref{ass_1} to hold}.

Now, we are ready to  prove  Theorem \ref{Thm_sbm}.  
\begin{proof}[Proof of  Theorem \ref{Thm_sbm}.]
	Using the assumption $q^*  = \|\n\|^{-\sigma+o(1)}$ from \eqref{ass_sbm},
	we observe  that 
	\[
	    \log(q^*\|\n\|) = (1+o(1)) (1-\sigma) \log\|\n\|. 
	\]
	Note that  all assumptions of  Theorems \ref{T:lower}  and    \ref{T:upper} hold
{	as they appear as the assumptions of    Theorem \ref{Thm_sbm}.}
	Also,  the quantities on the right hand sides of \eqref{eq_lower} and \eqref{eq_upper} satisfy
	\[
		\frac{ \log(q^*\|\n\|)}{\max_{i,j\in [k]}q_{ij}} \rightarrow \infty  \qquad \text{and} \qquad \|\n\|^{2- 4\sigma} \rightarrow \infty. 
	\]	
	Thus, applying Theorems \ref{T:lower}  and    \ref{T:upper}, we  get that, for any fixed $\eps>0$,
	whp
	\[
	       (1-\eps) \frac{w_*(\n)}{2 (1-\sigma)\log \|\n\|} \leq \chi(\G) \leq  (1+\eps) \frac{w_*(\n)}{2 (1-\sigma)\log \|\n\|}.  
	\]
	This completes the proof. 
\end{proof}


\subsection{Properties of $w(\cdot,Q)$ and $w_*(\cdot,Q)$}

In this section, we collect {some} facts about the quantaties $w(\cdot) = w(\cdot,Q)$ 
and  $ w_*(\cdot) =  w_*(\cdot,Q)$, defined by \eqref{w_def} and \eqref{def:mu-star}  for a general matrix $Q$.    
These properties are helpful for applications of Theorem \ref{Thm_sbm} and will also be repeatedely used in the proofs.

\begin{thm}\label{l:Qmatrix}
	Let $Q = (q_{ij})_{i,j \in [k]}$ be a  symmetric $k\times k$ matrix with non-negative entries.  
	Let $q^*$ and $\hat{q}(\cdot)$ be defined according \eqref{def:q-star}.
	Then, the following hold for any $\x=(x_1,\ldots,x_k)^T\in \pReals^k$.
	\begin{itemize}
		\item[\rm (a)]{{\rm[Scaling and monotonicity].}}	            If $\xvec' \in \pReals^k$ and $ \xvec'   \preceq s\xvec$ for some $s>0$, then
		$w(\xvec') \leq s w(\xvec)$ and $w_*(\xvec') \leq  s w_*(\xvec)$. 
		In particular, $w(s\xvec) = sw(\xvec)$ and   $w_*(s\xvec) = s w_*(\xvec)$. 
		\item[\rm(b)]{{\rm [Corner maximiser].}} There is $\zvec = (z_1,\ldots, z_k)^T$ with $z_i \in \{0,x_i\}$ for all $i\in [k]$
		such that 
		 \[{\dfrac{\zvec^T \, Q\,  \zvec}{\|\zvec\|} = w(\xvec):=\max_{\boldsymbol{0} \preceq \yvec \preceq \xvec}\ \dfrac{\yvec^T\, Q\, \yvec}{ \|\yvec\|}.}\]

		  \item[\rm(c)]{{\rm[Pseudodefinite property].}}  If $\yvec^T Q\yvec \geq 0$  for all $\yvec \in \Reals^k$ with $\sum_{i\in [k]}y_i =0$, then
		         	 \[{
		         	 w(\xvec) = w_*(\xvec):=  \inf_{\calS \in \mathcal{F}(\xvec)} \  \sum_{ \yvec \in \calS}\ w(\yvec,Q)}.
		         	 \]

		\item[\rm(d)]   {{\rm[Upper and lower bounds].}}  
		We have
		\[ 
		q^* \|\xvec\| \geq \hat{q}(\xvec) \|\xvec\|    \geq  w_*(\xvec) \geq  \frac{\left(\hat{q}(\xvec) \right)^2}{\sum_{i \in [k]}q_{ii}}\|\xvec\| \geq   \frac{\left(\hat{q}(\xvec) \right)^2}{kq^*}\|\xvec\|, 
		\]
		where  the lower bounds for $w_*(\xvec)$ hold  under the additional condition  that $q^*>0$.

		\item[\rm(e)]	{{\rm[Triangle inequality].}}  For any $\xvec' \in \pReals^k$, we have  
				$
		w_*(\xvec) + w_*(\xvec') \geq w_*(\xvec+\xvec').  
		$ 
	  

		\item[\rm(f)] {{\rm[Minimal system of $k$ vectors].}}  There exists a system of vectors $(\x^{(t)})_{t \in [k]}$,  each from  $ \pReals^k $,
		such that 
		$\sum_{t \in [k]} \x^{(t)} = \x$    	
		and  $\sum_{t \in [k]}  w(\xvec^{(t)}) = w_*(\xvec)$. 
		
		\item[\rm(g)]{{\rm[Near-optimal integer system]}.} If $\xvec \in \Naturals^k $ then  there exists a system of vectors $(\x^{(t)})_{t \in [k]}$,  each from  $ \Naturals^k $,
		such that 
		$\sum_{t \in [k]} \x^{(t)} = \x$    	 and 
		{ $\sum_{t \in [k]}  w(\xvec^{(t)}) \leq w_*(\xvec)      + k^2 q^*$.}  
	\end{itemize}
\end{thm}

The proof  of Theorem \ref{l:Qmatrix}  is technical and not very insightful, but for completeness it is provided at the end of the paper in Section \ref{S:Qmatrix}.  

\begin{remark} 
	 An interesting question not covered in this paper is how to  compute or at least approxiamte $w_*(\cdot)$ efficiently.  We give some examples in Section \ref{S:application}, but the question remains open in  general.  We believe that  {the  optimization problems of finding $w_*(\cdot)$ and $w(\cdot)$
	 	can be efficiently solved by fast converging iterative methods such as gradient descent and analogues of the simplex method.}
	  \end{remark}


\subsection{Structure of the rest of the paper}
 
 Section 3 covers applications of our main result, Theorem~\ref{Thm_sbm}.  
 We consider first the case of two blocks  in detail and then prove Theorems \ref{T:perc} and~\ref{T:main4}.  In addition, we study the unions of two independent random graphs from the stochastic block model. 
 In Section 4, we introduce the weighted independence number and prove the lower tail probability estimate of Theorem~\ref{T:lower}.
 Sections 5 and 6 are devoted to the upper tail probability estimate of Theorem \ref{T:upper}. 
 In Secton~5 we derive some preliminary estimates based on idea of separately colouring the blocks  of $\M(\n,P)$. In Section~6, we derive  an asymptotically optimal bound on the chromatic number using the estimates for the existence  of large weighted independent sets given in Section 4.2.
 Finally, we prove Theorem \ref{l:Qmatrix} in Section \ref{S:Qmatrix}.


\section{Applications of the main theorem}\label{S:application}
In this section we discuss some applications of Theorem~\ref{Thm_sbm}. 
Specifically, we consider  the case of  two blocks (Section~\ref{S:two}),  
the union of two independent random graphs from the stochastic  block model (Section~\ref{T:indun}),   
 percolations on a blow-up graph (Section~\ref{S:perc}), and  Chung-Lu model and its complement (Section~\ref{S:main4}).

\subsection{Two blocks}\label{S:two}

Consider the random graph $\G \sim \M(\n,P)$ with two blocks, where 
\[
k=2, \qquad
 \n = (n_1,n_2)^T \in \Naturals^2, \qquad   P =  \begin{bmatrix}
    p_{11}       & p_{12} \\
    p_{12}       & p_{22} 
\end{bmatrix} \quad \text{with}\ p_{12}=p_{21}.
\]
Let $B_1$ and $B_2$ denote the two blocks of $\G$, i.e., a partition of  the vertex set $V(\G)$, and let $\G_1:= \G[B_1] \sim\M(n_1,p_{11})$ and $\G_2:=\G[B_2]  \sim\M(n_2,p_{22})$ denote the  induced subgraphs of $\G \sim \SBM$ on $B_1$ and $B_2$, respectively.
Since $B_1$ and $B_2$ are disjoint, we have
\begin{equation}\label{eq:123}
	 \max \{\chi(\G_1), \chi(\G_2)\}\leq   \chi(\G) \leq \chi(\G_1)+\chi(\G_2). 
\end{equation}
For fixed  $p_{11},p_{22} \in (0,1)$, Martinsson et al. observed in  \cite[Section 4.1]{MPSM2020} that 
there are two threshold values $\underline{p}$ and $\overline{p} $ such that whp $\chi(\G)$ 
is asymptotically equal to the lower bound  of \eqref{eq:123}  if  $p_{12} \leq  \underline{p}$, but it is equal to the upper bound of \eqref{eq:123} if  $p_{12} \geq \overline{p} $.  Using Theorem \ref{Thm_sbm}, we extend this result to non-fixed  $p_{11} =p_{11}(n)$ or $p_{22}=p_{22}(n)$ that are allowed to vanish asymptotically. We also obtain the asymptotic formula for $\chi(\G)$ when $ \underline{p} \leq p_{12} \leq \overline{p}$.

To state our results, define  
\begin{align*}
		  &\overline{p} =  \overline{p} (p_{11},p_{22}) := 1 - (1-p_{11})^{\frac 12}(1-p_{22})^{\frac12},
\\
	&\underline{p}   =  \underline{p}(\n, p_{11},p_{22}) :=  1 - \min\left\{
	(1-p_{11})^{\frac12}\cdot {(1-p_{22})^{-\frac{n_2}{2n_1}}},  
	(1-p_{22})^{\frac12}\cdot {(1-p_{11})^{-\frac{n_1}{2n_2}}}\right\}.
\end{align*}
Obviously, $1 \geq \overline{p} \geq \underline{p}$ since $p_{11},p_{22} \in (0,1)$. 
Observe also $ \underline{p}\geq 0$ since
\begin{align*}
	\min\left\{
	\frac{(1-p_{11})^{\frac12}}{(1-p_{22})^{\frac{n_2}{2n_1}}},  
	\frac{(1-p_{22})^{\frac12}}{(1-{p_{11}})^{\frac{n_1}{2n_2}}}\right\} 
	\leq 
	 \left(
	     \frac{(1-p_{11})^{\frac12}}{(1-p_{22})^{\frac{n_2}{2n_1}}}
	 \right)^{\frac{n_1}{n_1+n_2}}
	 \left(
	     \frac{(1-p_{22})^{\frac12}}{(1-p_{11})^{\frac{n_1}{2n_2}}}
	 \right)^{\frac{n_2}{n_1+n_2}}
	  =1.
\end{align*}	
Recall from \eqref{Q_def} that  $Q = Q(P) ={(q_{ij})_{i,j \in \{1,2\}}}$ is defined by 
$
	q_{ij} :=  \log (\frac{1}{1-p_{ij}}).
$
\begin{thm}\label{C:two} 
Let   $\sigma \in [0, \sigma_0]$ for some fixed $0<\sigma_0 <\frac14$.  
Assume that   
\[
\|\n\| \rightarrow \infty, \qquad  
	  q_{11} = \|\n\|^{-\sigma+o(1)}, \qquad q_{22}  =  \|\n\|^{-\sigma+o(1)},
	\qquad \frac{q_{11}^2 }{q_{22}} +\frac{q_{22}^2 }{q_{11}} \ll \log \|\n\|.
\]
 Then the following hold whp.
  \begin{itemize}
  	  \item[(i)] If  $ \underline{p} \leq p_{12} \leq \overline{p}$ then
    \[
         \chi(\G) =
  		(1 + o(1)) \frac{ \n^T Q \n  }{ 2 (1-\sigma) \|\n\| 
  		\log \|\n\|} . 
      \]     
             \item[(ii)] If  $\overline{p} \leq p_{12} \leq 1$  then
          \[
               \chi(\G) 
 	 = 
              (1 + o(1)) \left( \chi(\G_1) + \chi(\G_2)\right)
              = (1 + o(1))  \frac{
  		  n_1 q_{11} +
  		 n_2 q_{22}     
            }{2 (1-\sigma) \log \|\n\|}.
          \]        
       \item[(iii)]  If  $0 \leq p_{12} \leq \underline{p}$  then
    \[
               \chi(\G) =
            (1 + o(1)) \max \left\{  \chi(\G_1) ,   \chi(\G_2) \right\}
=  		(1 + o(1))  \frac{\max \left\{
  		  n_1 q_{11} ,
  		 n_2 q_{22}            \right\}
            }{2 (1-\sigma) \log \|\n\|}.
         \]       
    \end{itemize}
\end{thm}
\begin{proof}
Let  $q^*$, $\hat{q}(\cdot)$, $w(\cdot,Q)$ and $w_*(\cdot,Q)$ be defined by \eqref{def:q-star}, \eqref{w_def}, and \eqref{def:mu-star}. We will first check that the assumptions of Theorem~\ref{Thm_sbm} are satisfied in part (i). To this end, note that  \eqref{ass_sbm} are given in Theorem~\ref{C:two}  and observe that 
 \begin{align*}
   p_{12}  &\leq  \overline{p}  \quad \Longleftrightarrow \quad q _{12} \leq  \dfrac{1}{2} q_{11} + \dfrac12 q_{22}.
  \end{align*}
   In particular, we get that  $q_{12}\leq q^*$, thus
   \[
   	 \left(1+ \dfrac{1}{q^*}\right) \max_{i,j\in \{1,2\}} q_{ij} \leq q^* + 1 
   	 \leq \frac{(q^*)^2}{ \hat{q}(\n)} +1
   	 \leq \frac{q_{11}^2}{q_{22}}
 +    \frac{q_{22}^2}{q_{11}}  +1 \ll \log \|\n\|.
 \]
Using also{the bounds of Theorem \ref{l:Qmatrix}(d)}, we find that
   \[
   	 w_*(\n,Q) \geq \frac{(\hat{q}(\n))^2}{ 2 q^*} \|\n\| \gg \frac{\hat{q}(\n) q^* \|\n\|}{ \log \|\n\|}. 
   \]
   This establishes  \eqref{ass_1} and  \eqref{ass_2}.

   To prove part (i) by applying Theorem~\ref{Thm_sbm}, it remains to show that if  $ \underline{p} \leq p_{12} \leq \overline{p}$ then
     $$w_*(\n, Q)= \frac{\n^TQ \n}{\|\n\|}.$$
      The  inequality $q _{12} \leq  \dfrac{1}{2} q_{11} + \dfrac12 q_{22}$  is  also equivalent to  $\yvec^TQ\yvec \geq 0$ for any $\yvec \in \Reals^2$ with $y_{11}+y_{22}= 0$ (clearly, one only needs to consider $\yvec = (1,-1)^T$).  Using{the corner maximiser and the pseudodefinite properties in  Theorem} \ref{l:Qmatrix}(b,c), we get that 
  \[
  	w_*(\n, Q) = w(\n, Q) = \max\left\{ n_1 q_{11}, n_2 q_{22},   \dfrac{ \n^T Q \n  }{  \|\n\| }    
  		  \right\}.
  \]
Observe that 
  \[
  	n_1 q_{11} \leq  \frac{ \n^T Q \n  }{  \|\n\| } =  \frac{n_1^2 q_{11} + 2 n_1 n_2 q_{12} + n_2^2 q_{22} }{ n_1+n_2}
  \]
 holds  whenever $  q_{12}  \geq \frac12 q_{11} -  \frac{n_2}{2n_1} q_{22}$. Similarly, 
    $n_2 q_{22} \leq  \frac{ \n^T Q \n  }{  \|\n\| }$ holds whenever $q_{12}  \geq \frac12 q_{22} -  \frac{n_1}{n_2} q_{11}$. 
  Next, we recall the second assumption of part (i) that   
  $ p_{12}  \geq  \underline{p}$, which is equivalent to 
      \[
     q _{12} \geq   \max\left\{\dfrac12 q_{11} -  \dfrac{n_2}{2n_1}q_{22}, \dfrac12 q_{22} -  \dfrac{n_1}{2n_2}q_{11} \right\}.  
   \]  
   Thus, we conclude that
       \[
       		w_*(\n, Q) = \max\left\{ n_1 q_{11}, n_2 q_{22},   \dfrac{ \n^T Q \n  }{  \|\n\| } \right\} =   \dfrac{ \n^T Q \n  }{  \|\n\|}, 
       \]
      which  completes the proof of (i).

For part (ii), we consider the random graph $\overline{\G} \sim \M(n,\overline{P})$ such that 
$\overline{\G} \subset \G$, where  the diagonal entries  of $\overline{P}$ are  the same  as of $P$ while  
the off-digonal entries of $\overline{P}$ equal $\overline{p}$. 
Using \eqref{eq:123}, we get that
\[
	\chi(\overline{\G}) \leq \chi(\G) \leq \chi(\G_1) +\chi(\G_2).
\]
Thus, it is sufficient to show that  whp
\begin{align} 
	 \chi(\overline{\G}) &= (1+o(1))  \frac{
  		  n_1 q_{11} +
  		 n_2 q_{22}     
            }{2 (1-\sigma) \log \|\n\|}, 
            \label{part2_1}
            \\
            \chi(\G_1)+ \chi(\G_2)   &= (1+o(1))  \frac{
  		  n_1 q_{11} +
  		 n_2 q_{22}     
            }{2 (1-\sigma) \log \|\n\|}.
            \label{part2_2}
\end{align}

 Applying part (i) to $\overline{\G}$, we get that 
\[
	  \chi(\overline{\G}) =
  		(1 + o(1)) \frac{ \n^T \overline{Q}\n  }{ 2 (1-\sigma) \|\n\| 
  		\log \|\n\|},
\]
where $\overline{Q}$ is the matrix corresponding to $\overline{P}$. 
Note that
\[
	\frac{\n^T \overline{Q}\n}{\|\n\|} =  
	\frac{n_1^2 q_{11} +  n_1 n_2 (q_{11}+q_{22})  + n_2^2 q_{22} }{ n_1+n_2} = 
  		  n_1 q_{11} +
  		 n_2 q_{22}.
\]
Thus, \eqref{part2_1} holds.
  
 Next, observe that  $\eqref{part2_2}$ is implied by Theorem \ref{T:oneblock} if  $  n_{1} = \|\n\|^{1 +o(1)}  $ 
 and    $  n_{2 } = \|\n\|^{1 +o(1)}$. Otherwise, if one of the parts is very small, say $\G_1$, then we have
 \[
   \|\n\| = (1+o(1))n_2 \quad \text{and} \quad
 		 n_1 q_{11} +
  		 n_2 q_{22} = (1+o(1)) n_2 q_{22}.
 \]
 Applying Theorem \ref{T:oneblock} to $\G_2$,  we get whp
 \[
 	\chi(\G_2) = (1+o(1)) \frac{n_2 q_{22}}{ 2\log (p_{22} n_2)} =  (1+o(1)) \frac{n_1q_{11} + n_2 q_{22}}{ 2(1-\sigma) \log \|\n\|}.
 \]
 Let  $n_1' = \frac{n_2 q_{22}}{ q_{11} \log \|\n\|}$. By the assumptions,  $n_1' = \|\n\|^{1+o(1)} \gg n_1$.
Using the embedding   $\G_1 \subset \G(n_1', p_{11})$, we estimate
 \[
 	\chi(\G_1) \leq \chi(\G(n_1', p_{11})) =(1+o(1))  \frac{n_1' q_{11}}{ 2\log (p_{11} n_1')} = o(1) \frac{n_1q_{11} + n_2 q_{22}}{  \log \|\n\|} .
 \]
 The above two bounds for $\chi(\G_1) $ and  $\chi(\G_2)$ prove \eqref{part2_2}, completing the proof of part (ii).

 Part (iii) is proved in a similar way to part (ii).

\end{proof}


 \subsection{Union of two independent random graphs}\label{S:indun}
 
 Consider two independent binomial random graphs $\G_1 = \G(n,p_1)$ and $\G_2 = \G(n,p_2)$ on the same vertex  set $[n]$,
 where $p_1,p_2$ are some constants from $(0,1)$. 
   It is easy to show that  their union 
 $\G_1 \cup \G_2$ is also a binomial random graph  $\G(n,p)$, where  $1- p =  (1-p_1)(1-p_2).$ This is equivalent to 
 \begin{equation}\label{eq:p12}
 	\log \left(\dfrac{1}{1-p}\right) =  	\log \left(\dfrac{1}{1-p_1}\right)  +	\log \left(\dfrac{1}{1-p_2}\right). 
 \end{equation}
 Then, by formula \eqref{eq:Gnp},  we get that whp
 \begin{equation}\label{eq:indun}
 	\chi(\G_1 \cup \G_2) = (1+o(1)) \left(\chi(\G_1) + \chi(\G_2)\right).
 \end{equation}
 That is, the chromatic number of the union of two independent random graphs is whp asymptotically equal  to the sum of the chromatic numbers of the two binomial random graphs.  In this section, we  prove a generalisation of  this observation to the stochastic block model{based on Theorem~\ref{Thm_sbm}. Apart from the assumptions of Theorem~\ref{Thm_sbm}, we also insist that both random graph models satisfy the pseudodefinite property of Theorem \ref{l:Qmatrix}(c).}

 \begin{thm}\label{T:indun}
 	Let $\G_1 \sim \calG(\n,P_1)$ and $\G_2 \sim \calG(\n,P_2)$ be independent random graphs from the stochastic block models, where $P_1$ and $P_2$ satisfy the assumptions of Theorem \ref{Thm_sbm} (with the same $\sigma$).
 	Assume also that  $\yvec^T Q_1\yvec \geq  0$ and $\yvec^T Q_2\yvec \geq  0$ for all $\yvec \in \Reals^k$ with $y_1+\ldots +y_k =0$,
 	where $Q_1 = Q(P_1)$ and $Q_2 = Q(P_2)$ are defined by \eqref{Q_def}.
 	 Then, whp 
 	 	\begin{equation}\label{eq1:G12}
 		\chi(\G_1 \cup \G_2) \leq (1+o(1)) \left(\chi(\G_1) + \chi(\G_2)\right).
 	\end{equation}
 	In addition, if  	$w(\n,Q_1) =  \frac{\n^T Q_1\n}{\|\n\|}$
 	and $w(\n,Q_2) =  \frac{\n^T Q_2\n}{\|\n\|}$ then  \eqref{eq:indun} holds whp.
 \end{thm}
 \begin{proof}
   Applying Theorem \ref{Thm_sbm}, we find that  whp
   \begin{equation}\label{chi12}
   	\chi(\G_1) =  (1+o(1))\frac{w_*(\n,Q_1)}{ 2(1-\sigma) \log \|\n\|}, \qquad  \chi(\G_2) = (1+o(1))\frac{w_*(\n,Q_2)}{ 2(1-\sigma) \log \|\n\|}.
   \end{equation}
  Observe that the union $\G_1 \cup \G_2$ also belongs to the stochastic block model $\mathcal G(\n,P)${with the entries of $P$ defined  similarly to 
  \eqref{eq:p12}. Observe that} 
   \begin{equation*}
   	Q= Q_1 +Q_2,
   \end{equation*}
   where $Q= Q(P)$ is defined by \eqref{Q_def}. In particular, we get that 
   $w_*(\n, Q) \geq w_*(\n,Q_1)$ and $w_*(\n,Q_2)$. It is straightforward  to check that $P$ and $Q$ satisfy 
  the assumptions of Theorem \ref{Thm_sbm} (with the same $\sigma$). Thus, we get whp
  \begin{equation}\label{chiun}
  	\chi(\G_1 \cup \G_2) =  (1+o(1))\frac{w_*(\n,Q_1 +Q_2)}{ 2(1-\sigma) \log \|\n\|}. 
  \end{equation}
  Next, by {the pseudodefinite property in Theorem} \ref{l:Qmatrix}(c), we get that
  \[	
  	w_*(\n,Q_1) = w(\n, Q_1), \qquad w_*(\n,Q_2) = w(\n, Q_2).
  \]
  Note also  $\yvec^T Q\yvec = \yvec^T Q_1\yvec+ \yvec^T Q_2\yvec \geq  0$ for all $\yvec \in \Reals^n$ with $y_1+\ldots +y_n =0$. 
  Using  Theorem~\ref{l:Qmatrix}(c) again, we find that
  \begin{align*}
  		w_*(\n,Q_1+Q_2) &= w(\n, Q_1+Q_2)\\
  		&= {\max_{\boldsymbol{0} \preceq \yvec  \preceq \n} \dfrac{\yvec^T (Q_1+Q_2) \yvec }{\|\yvec\|}
  		 \leq \max_{\boldsymbol{0} \preceq \yvec  \preceq \n} \dfrac{\yvec^T Q_1 \yvec }{\|\yvec\|}
  		+\max_{\boldsymbol{0} \preceq \yvec  \preceq \n} \dfrac{\yvec^T Q_2 \yvec }{\|\yvec\|}}
  		\\
  		&= w(\n,Q_1) + w(\n,Q_2) = w_*(\n,Q_1) + w_*(\n,Q_2).
  \end{align*}
  Combining the above, we prove \eqref{eq1:G12}. 
  
 To establish \eqref{eq:indun}   under the additional conditions that 	$w(\n,Q_1) =  \frac{\n^T Q_1\n}{\|\n\|}$
 	and $w(\n,Q_2) =  \frac{\n^T Q_2\n}{\|\n\|}$,   we will show that 
  \[
  	w_*(\n,Q_1+Q_2) =  w_*(\n,Q_1)+ w_*(\n,Q_2).
  \]
  Then, the result would follow by    \eqref{chi12}  and \eqref{chiun}.
  We already proved that 
  $
  	w_*(\n,Q_1) = w(\n,Q_1)$,  $ w_*(\n,Q_2) = w(\n,Q_2)$, $w_*(\n,Q_1+Q_2) = w(\n,Q_1+ Q_2)$,
   and  $w(\n, Q_1 +Q_2) \leq  w(\n,Q_1) + w(\n,Q_2).$
  Thus, it remains to prove that
  \begin{equation}\label{last_inequality}
  	w(\n, Q_1 +Q_2) \geq w(\n,Q_1) + w(\n,Q_2). 
  \end{equation}
  Note that 
  \[
  	\dfrac{\n^T Q_1\n}{\|\n\|} + \dfrac{\n^T Q_2\n}{\|\n\|}
  	= \dfrac{\n^T (Q_1+Q_2)\n}{\|\n\|}
  	\leq {\max_{\boldsymbol{0} \preceq \yvec  \preceq \n}} \dfrac{\yvec^T (Q_1+Q_2) \yvec }{\|\yvec\|} = w(\n, Q_1+Q_2).
  \]
  Recalling the assumptions that   	$w(\n,Q_1) =  \frac{\n^T Q_1\n}{\|\n\|}$
 	and $w(\n,Q_2) =  \frac{\n^T Q_2\n}{\|\n\|}$,   we derive \eqref{last_inequality}, thus completing the proof.
 \end{proof}
 

\subsection{Percolations on blow-up graphs: proof of Theorem \ref{T:perc}}\label{S:perc}
 
In order to prove  Theorem~\ref{T:perc} by applying Theorem~\ref{Thm_sbm}, we need the following auxiliary result on  the chromatic number of any deterministic graph which can be found as a solution of a discrete optimisation problem similar to   \eqref{def:mu-star}.  

For a graph $G$, let $\mad(G)$ denote  \emph{the  	maximum average degree} over all subgraphs of $G$.

\begin{lemma}\label{L1:perc}
For any graph $G$, we have 
	\begin{equation*}
	\chi(G) = \min_{\mathcal{U}} \sum\nolimits_{S \in \mathcal{U}} (1+\mad(G[S])),
	\end{equation*}
	where the  minimum is over all partitions  $\mathcal{U}$ of the vertex set $V(G)$ and  $G[S]$ denotes the induced subgraph of $G$.
\end{lemma}

\begin{proof}
It is a standard  fact from the graph theory that 
\begin{equation}\label{chi:upper}
\chi(G) \leq  1+ \max_{U \subseteq V(G)} \delta_G (U), 
\end{equation}
where $\delta_G(U)$ is the minimum  degree in the induced graph $G[U]$. 
The proof  of  \eqref{chi:upper} is by a straightforward induction  on $|V(G)|$; see, for example,
\cite[Lemma 7.12]{FK2015}.

Clearly, we have that 
\[
	\max_{U \subseteq V(G)} \delta_G (U) \leq \mad(G),
\]
{which together with \eqref{chi:upper} implies that, for any 
$S\in V(G)$,  
\[
\chi(G[S]) \leq 1 + \mad(G[S]).
\]}
Thus,  for any partition  $\mathcal{U}$ of $V(G)$, we get  
\begin{equation*}
\chi(G) \leq  \sum\nolimits_{S \in \mathcal{U}} \chi(G[S]) \leq \sum\nolimits_{S \in \mathcal{U}} (1+\mad(G[S]))
\end{equation*}
by colouring all parts of $\mathcal{U}$ in different colours.

On the other hand, for the partition $\mathcal{U}$ of $V(G)$  corresponding {to}   the colour classes of an optimal colouring of $G$, we observe that  $\mad(G[S]) = 0$ for all $S \in \mathcal{U}$.  Thus, we get that 
\[
	 \sum\nolimits_{S \in \mathcal{U}} (1+\mad(G[S])) = \chi(G).
\] 
This completes the proof.
\end{proof}

We  proceed to the proof of Theorem \ref{T:perc}.  
{The three key proof ingredients   are  the following:
\begin{itemize}
		\item the asymptotic formula for the chromatic number of the stochcastic block model in terms of $w_*(\cdot)$ given in Theorem~\ref{Thm_sbm};
		\item  Lemma \ref{L1:perc} that expresses  the chromatic number of an arbitrary graph in  terms of  the maximum average degree that is similar to the underlying  optimisation problem for  $w_*(\cdot)$;
		\item the existence of a small near-optimal integer system given by Theorem \ref{l:Qmatrix}(g) that approximates $w_*(\cdot)$.
\end{itemize}}
\begin{proof}[Proof of  Theorem \ref{T:perc}.]
	Let $\epsilon \in (0,\frac14)$ and $H$ be a graph on vertex set $[k]$. Assume that $p=p(n)$ satisfies the conditions in Theorem \ref{T:perc}. 
Let $A$ denote the adjacency matrix of $H$ and $I$ be the $k\times k$ identity matrix.

First, we note that the percolated random graph $G_p$, where $G = G_H(\n)$,  is distributed according to  $\M(\n,P)$ with $P = p(I+A)$. Let $Q=Q(P)$ be defined according to \eqref{Q_def}.
We apply Theorem~\ref{Thm_sbm} with $\sigma:=  -\frac{\log p} {\log \|\n\|} \leq \sigma_0:=\frac{1}{4} -\epsilon$. All assumptions of Theorem \ref{Thm_sbm} are straightforward to  check.   {Since $(1-\sigma)\log \|\n\|=\log(p\|\n\|)$ by definition of $\sigma$,}  Theorem~\ref{Thm_sbm} implies that 
\[
	\chi(G_p) {= 
		(1+o(1))\frac{w_*(\n,Q)}{ 2(1-\sigma)\log \|\n\| }} = (1+o(1))  \frac{w_*(\n,Q)}{  2\log  (p\|\n\|)}.
\] 
Note that all elements of matrix $Q$ are $\log (\tfrac{1}{1-p})$ or $0$.  More precisely,  
$Q =   \log (\tfrac{1}{1-p})  \tilde{Q}$, where $\tilde{Q}:= I+A$.  
By {the scaling property in Theorem} \ref{l:Qmatrix}(a), we have
        \[
          w_*(\n,Q) = \log (\tfrac{1}{1-p})\ w_*(\n,\tilde{Q}).
         \] 
     Thus, to prove  \eqref{eq:perc}, it remains to show that
      \begin{equation}\label{last:perc}
      \chi(G) = (1+o(1))  w_*(\n,\tilde{Q})
      \end{equation}
   
       To show \eqref{last:perc}, we employ Lemma \ref{L1:perc}. {To this end,  for    any $S \subseteq V(G)$,   we  define 
            	\begin{equation}\label{def:b(U)} 
     \b(S) = (b_1(S), \ldots, b_k(S))^T \in \Naturals^k \quad \text{with}\quad  		b_i(S):= |S\cap B_i| \quad\text{for }  i\in[k],
       	\end{equation} and  observe that, for any $U \subseteq S$  
 	\begin{equation}\label{eq:b(S)}
		\begin{aligned} 
	   \b(U)^T \tilde{Q} \b(U) &=   
	     \sum_{i \in [k]}   b_i(U)^2  +  2 \sum_{ij \in H} b_i(U) b_j(U) \\&  
	 =   \|\b(U)\|+ 2 \left( \sum_{i \in [k]}  \dfrac{b_i(U)(b_i(U)-1)}{2} +   \sum_{ij \in H} b_i(U) b_j(U)\right)\\&= 
	      |U|+ 2|E(G[U])|.
	   \end{aligned}
	\end{equation}  }
	  
	   Using \eqref{eq:b(S)} and {the corner maximiser property} in Theorem \ref{l:Qmatrix}(b), we find that 
	 \[
	 w(\b(S), \tilde{Q}) {=  \max_{U \subseteq S }\ \frac{\b(U)^T \tilde{Q}\ \b(U)}{ \|\b(U)\|} = \max_{U \subseteq S}  \frac{|U|+ 2|E(G[U])|}{ |U|} =} 1+\mad(G[S]). 
	 \] 
	 {By the definition of $w_*(\n,\tilde{Q})$,  we obtain that 
	 	\begin{equation}\label{eq:w*-partition}
	 		w_*(\n,\tilde{Q}) \leq  \min_{\mathcal{U}} \sum\nolimits_{S \in \mathcal{U}} (1+\mad(G[S])), 
	 	\end{equation}
	 	where the   minimum is over all partitions  $\mathcal{U}$ of the vertex set $V(G)$.
	 	On the other hand, due to  the near-optimal integer system of Theorem \ref{l:Qmatrix}(g), there exists  a partition $\mathcal{U^*}$ of $V(G)$
	 	such that 
	 	\[
	 	 \min_{\mathcal{U}} \sum\nolimits_{S \in \mathcal{U}} (1+\mad(G[S])) \leq    \sum\nolimits_{S \in \mathcal{U^*}} (1+\mad(G[S]))
	 	\leq  w_*(\n,\tilde{Q}) +  k^2 \tilde{q}^*.
	 	\] 
	 	{This together with \eqref{eq:w*-partition} gives 
	 	\[ 	
	 	w_*(\n,\tilde{Q}) \leq  \min_{\mathcal{U}} \sum\nolimits_{S \in \mathcal{U}} (1+\mad(G[S])) 	\leq  w_*(\n,\tilde{Q}) +  k^2 \tilde{q}^*.
	 		\] 	
 	}
	 	Note that the bounds of  Theorem \ref{l:Qmatrix}(d) imply $w_*(\n,\tilde{Q}) \rightarrow \infty$. Because  $\tilde{q}^*=1$ and $k=|V(H)|$ is a fixed constant,  
	 	using Lemma \ref{L1:perc}, we get that 
	 	\begin{equation*}
	 		\chi(G) = \min_{\mathcal{U}} \sum\nolimits_{S \in \mathcal{U}} (1+\mad(G[S])) 
	 		=  w_*(\n,\tilde{Q}) +O(1), 
	 	\end{equation*}
	 	which implies \eqref{last:perc} and completes the proof.  }
 \end{proof}


 \subsection{Chung-Lu model: proof of Theorem \ref{T:main4}}\label{S:main4}
 {Let $k = k(n) \in \Naturals$ be such that $1\ll k \ll \log n$.}
 Let $S_1 = [0, \frac1k],\  S_2 = (\frac1k, \frac2k], \ldots,  S_k = (\frac{k-1}{k},1]$. Define  
 $\n(\uvec) = (n_1,\ldots, n_k)^T$ by 
 \[
 	n_i = n_i(\uvec) := \left|\{t \in [n] \st u_t \in S_i\}\right|.
 \]
 Define two $k \times k$ matrices $P^L = (p_{ij}^L)_{i,j \in [k]}$ and $P^U = (p_{ij}^U)_{i,j \in [k]}$ by 
 \[
 	p_{ij}^L := p\cdot \dfrac{(i-1)(j-1)}{k^2}, \qquad  	p_{ij}^U :=  p\cdot\dfrac{ij}{k^2}. 
 \]
 Then, for any two vertices $a,b \in V(\G_p^{\times}) = [n]$, we have 
 \[
 	p_{ij}^L \leq p_{ab}^{\times} \leq  p_{ij}^U,
 \]
 where $i=i(a)$ and $j = j(b)$ are such that $u_a \in S_i$ and $u_b \in S_j$. 
  Therefore, there are two random graphs  $\G^L \sim \mathcal G(\n(\uvec),P^L )$ and 
  $\G^U \sim \mathcal G(\n(\uvec),P^U )$ such that
 $\G^L \subseteq \G_p^{\times} \subseteq \G^U$.  
  Furthermore, we find that 
 \[
 		\chi\left(\G^{U}\right) \leq  \chi\left( \G_p^{\times}\right)  \leq  \chi\left(\G^{L}\right).
 \]
{ Let $Q^L= Q(P^L)$ and $Q^U = Q(P^U)$ be defined according to \eqref{Q_def}. }

 { Next, we show that $w_*\left(\n(\uvec), Q^L\right)  = \Omega(n)$ 
 	and $w_*\left(\n(\uvec), Q^U\right) = \Omega(n)$. Then
 	the assumptions of Theorem \ref{Thm_sbm} hold for both random graphs $\G^L $ and $\G^U$.
Indeed, setting  $\sigma:= -\dfrac{\log p}{\log n}$ 
and using the assumptions of Theorem \ref{T:main4}, that is, $1\gg p \geq n^{-1/4+\epsilon}$ and $\sum_{i\in[n]} u_i = \Omega(n)$, we get 
\[
q^*, \hat{q}((\n(\uvec)) = \Theta(p) = \Theta(n^{-\sigma}) =  n^{-\sigma+o(1)} \ll 1
\]
  for $\G^L $ and $\G^U$.
Recalling  also $k\ll \log n$,
we obtain \eqref{ass_sbm} and \eqref{ass_1}. Finally \eqref{ass_2}
 holds if   $w_*\left(\n(\uvec), Q^L\right) $
and $w_*\left(\n(\uvec), Q^U\right) $ are $\Omega(n)$ since 
$k \hat{q}(\n(\uvec)) q^* \ll \log n$.
Thus, to complete the proof of  Theorem~\ref{T:main4}(a), it remains to establish  the following lemma.
 }

 \begin{lemma}\label{L:w-star-star} 
  Suppose  the assumptions of Theorem \ref{T:main4} hold.
   Then
 	\begin{align*}
 	    \max_{U \subseteq [n]} \dfrac{1}{|U|} \left(\sum\nolimits_{t \in U} u_t\right)^2
 	=   \dfrac{1+o(1)}{p} w_*\left(\n(\uvec), Q^L\right)  
    = \dfrac{1+o(1)}{p}
 		w_*\left(\n(\uvec), Q^U\right)   = \Omega(n).
 	\end{align*}
 \end{lemma}
 \begin{proof}
       Since  $\sum_{t \in [n]} u_t = \Omega( n)$, we  find that
  \begin{equation}\label{pn-bound}
  	M  := \max_{U \subseteq [n]} \frac{1}{|U|} \left(\sum\nolimits_{t \in U} u_t\right)^2 
  	\geq \frac{1}{n}\left(\sum\nolimits_{t \in [n]} u_t\right)^2  =  \Omega(n).
  \end{equation}
  Since $p =o(1)$, we have  
  $\log \frac{1}{1-p xy} =  (1+o(1)) p xy$  uniformly over all $x,y \in [0,1]$. 
  Then,  {by the definition of $w_*(\cdot)$}, we derive that  
 \begin{align*}
 	w_*\left(\n(\uvec), Q^{L}\right) = (1+o(1))   w_*\left(\n(\uvec), P^L\right), \\
 	w_*\left(\n(\uvec), Q^{U}\right) = (1+o(1))   w_*\left(\n(\uvec),P^U\right).
 \end{align*}
 For any $\xvec \in \Reals^k$, we have
 \begin{equation}\label{eq:Pl}
 	\xvec^T P^L\xvec = \dfrac{p}{k^2}\left(\sum\nolimits_{i \in [k]} x_i (i-1)\right)^2 \geq 0.
 \end{equation}
Using  {the pseudodefinite property in Theorem} \ref{l:Qmatrix}(c), we find that 
$
	w_*\left(\n(\uvec),P^L\right) = w\left(\n(\uvec),P^L\right).
$
Similarly, we get that $
	w_*\left(\n(\uvec),P^U\right) = w\left(\n(\uvec),P^U\right).
$  
  Thus, it remains to  show that 
   \begin{equation}\label{M-remain}
   	M
   	 =  \dfrac{1+o(1)}{p}  w\left(\n(\uvec),P^L\right) =  \dfrac{1+o(1)}{p}  w\left(\n(\uvec),P^U\right).
   \end{equation}
   
     Let $U^* \subseteq [n]$ be the set that maximises   
     $\frac{1}{|U|} \left(\sum\nolimits_{t \in U} u_t\right)^2 $, that is,
    { \[
     	M = \frac{1}{|U^*|} \left(\sum\nolimits_{t \in U*} u_t\right)^2.
     \]}
      Using  the trivial bound 
     $\sum\nolimits_{t\in U^*} u_t \leq |U^*|$ and \eqref{pn-bound}, we get that
  \begin{equation}\label{eq:eps2n}
  	\sum\nolimits_{t \in U^*} u_t \geq \frac{1}{|U^*|} \left(\sum\nolimits_{t \in U^*} u_t\right)^2  = {\Omega(n)}.
  \end{equation}
  Let  
  \begin{align*}
  \xvec(U^*):= (x_1,\ldots,x_k)^T \in \Naturals^k, \qquad
  x_i := |\{ t \in [n] \st u_t \in S_i \cap U^*\}.
  \end{align*}
 Combining \eqref{eq:eps2n} and the trivial bound $\|\xvec(U^*)\| = |U^*|\leq n$, we get that 
   \[
   		\sum\nolimits_{t \in U^*} u_t   = (1+O(k^{-1})) \sum_{i \in [k]} \frac{i-1}{k}  x_i = (1+O(k^{-1})) \sum_{i \in [k]} \frac{i}{k} x_i.
   \]
{Due to \eqref{eq:Pl} and  a similar formula for $P^U$, we get that 
	$(\sum\nolimits_{t \in U^*} u_t)^2 $ is equivalent to   $\xvec^T(U^*) P^L\xvec(U^*)$
	and $\xvec^T(U^*) P^U \xvec(U^*)$ up to the factor $p$. 
Recalling  $\|\xvec(U^*)\| = |U^*|$, we get that}
\[
  	 pM
  	 = (1+o(1)) \frac{\xvec^T(U^*) P^L\xvec(U^*)}{\|\xvec(U^*)\|} 
  	 =  (1+o(1)) \frac{\xvec^T(U^*) P^U\xvec(U^*)}{\|\xvec(U^*)\|}. 
   \]
   This implies  
   \[
   {	w\left(\n(\uvec),P^U\right) \geq w\left(\n(\uvec),P^L\right) \geq (1+o(1))pM.}
   \]
   
    For the other direction, using  {the corner maximiser property} in Theorem   \ref{l:Qmatrix}(b), we get that there is $W \subseteq [k]$ such that
    \[
    	w\left(\n(\uvec),P^U\right) =  \frac{p}{\sum\nolimits_{i \in W} n_i} \left(\sum_{i \in W}  \frac{i}{k} n_i  \right)^2.
    \]
    Let $U(W):= \{t \in [n] \st u_t \in \cup_{i\in W} S_i\}$. Then, $\sum_{i \in W} n_i = |U(W)|$.
    We also have 
    \[
    		\sum_{i \in W}  \frac{i-1}{k} n_i  \leq  \sum_{t \in U(W)} u_t \leq \sum_{i \in W}  \frac{i}{k} n_i. 
    \]
   We already established that  $w\left(\n(\uvec),P^U\right) \geq (1+o(1))  pM= \Omega ( p n)$.  Thus, 
   \[
   	\sum_{i \in W}  \frac{i}{k} n_i \geq \frac{1}{\sum\nolimits_{i \in W} n_i} \left(\sum_{i \in W}  \frac{i}{k} n_i  \right)^2  = \Omega(n).
   \]  
   Therefore, 
   \[
    \frac{1}{\sum\nolimits_{i \in W} n_i} \left(\sum_{i \in W}  \frac{i}{k} n_i  \right)^2
   	= (1+O(k^{-1})) \frac{1}{|U|} \left(\sum_{t \in U} u_t\right)^2.
   \]
   This implies 
   \[
   	w\left(\n(\uvec),P^L\right) \leq w\left(\n(\uvec),P^U\right) \leq (1+o(1))pM.
   \]
   This  
   completes the proof of  required bound \eqref{M-remain} and of  the lemma.
  \end{proof}

 We proceed to  the proof of Theorem \ref{T:main4}(b).
  Define two $k \times k$ matrices $\widehat{P}^L = (\widehat{p}_{ij}^L)_{i,j \in [k]}$ and $\widehat{P}^U = (\widehat{p}_{ij}^U)_{i,j \in [k]}$ by 
 \[
 	\widehat{p}_{ij}^L := p\cdot \dfrac{(i-1)+(j-1)}{k}, \qquad  \widehat{p}_{ij}^U :=  p\cdot\dfrac{i+j}{k}. 
 \]
 Then, for any two vertices $a,b \in V(\G_p^{+}) = [n]$, we have 
 \[
 	\widehat{p}_{ij}^L \leq p_{ab}^{+} \leq  \widehat{p}_{ij}^U,
 \]
 where $i=i(a)$ and $j = j(b)$ are such that $u_a \in S_i$ and $u_b \in S_j$. 
  Therefore, there are two random graphs  $\widehat{\G}^L \sim \mathcal G(\n(\uvec),\widehat{P}^L )$ and 
  $\widehat{\G}^U \sim \mathcal G(\n(\uvec),\widehat{P}^U )$ such that
 $\widehat{\G}^L \subseteq \G_p^{+} \subseteq \widehat{\G}^U$.  
  Furthermore, we find that 
 \[
 		\chi\left(\widehat{\G}^{U}\right) \leq  \chi\left( \G_p^{+}\right)  \leq  \chi\left(\widehat{\G}^{L}\right).
 \]
Then  Theorem \ref{T:main4}(b) follows  immediately by combining  Theorem \ref{Thm_sbm} and  the following lemma. 
\begin{lemma}\label{L:w-star-star2} Let the assumptions of Theorem \ref{T:main4} hold.
  Let $\widehat{Q}^L= Q(\widehat{P}^L)$ and $Q^U = Q(\widehat{P}^U)$ be defined according to \eqref{Q_def}.  Then
 	\begin{align*}
 	      \sum_{t\in [n]}u_t
 	=   \dfrac{1+o(1)}{p}	w_*\left(\n(\uvec), \widehat{Q}^{L}\right)  
    =\dfrac{1+o(1)}{p} 		w_*\left(\n(\uvec), \widehat{Q}^{U}\right)   = \Omega(n).
 	\end{align*}
 \end{lemma}
 \begin{proof}
 		  Since $p =o(1)$, we have $\log \frac{1}{1-p (x+y)} =  (1+o(1)) p (x+y)$ uniformly over 
 		  $x,y \in [0,1]$. Then,  {by the definition of $w_*(\cdot)$}, we observe  that  
 \begin{align*}
 	w_*\left(\n(\uvec), \widehat{Q}^{L}\right) = (1+o(1))   w_*\left(\n(\uvec),\widehat{P}^{L}\right), \\
 	w_*\left(\n(\uvec),\widehat{Q}^{U}\right) = (1+o(1))   w_*\left(\n(\uvec),\widehat{P}^{U}\right).
 \end{align*}
 For any $\xvec \in \pReals^k$, we have
 \[
 	w\left(\xvec, \widehat{P}^{L}\right)  = {\max_{\boldsymbol{0}\prec \yvec \prec \xvec}} \frac{\yvec^T \widehat{P}^{L}\yvec}{\|\yvec\|} = p \cdot {\max_{\boldsymbol{0}\prec \yvec \prec \xvec}}  \sum_{i \in [k]}  \frac{i-1}{k} y_i
 	=p \cdot \sum_{i \in [k]}  \frac{i-1}{k} x_i.
 \] 
 By the  definition of $w_*(\cdot)$, we find that 
 \[
 w_*\left(\n(\uvec), \widehat{P}^L\right) =  p\cdot \sum_{i \in [k]}  \frac{i-1}{k} n_i.
 \]
Observing that  
\[ 
 \sum_{i \in [k]}  \frac{i}{k} n_i \geq  \sum_{t\in [n]} u_t \geq   \sum_{i \in [k]}  \frac{i-1}{k} n_i  
 \]
 and recalling $ \sum_{t\in [n]} u_t = {\Omega(n)}$,   we derive
 \[
 	w_*\left(\n(\uvec), \widehat{P}^L\right) = (1+O(k^{-1})) 
 	 p \cdot\sum\nolimits_{t\in [n]}u_t.
 \]
 Similarly, we prove $w_*\left(\n(\uvec), \widehat{P}^U\right) =  (1+O(k^{-1})) p \cdot
 	 \sum_{t\in [n]}u_t.$ This completes the proof.
 \end{proof}


\section{Weighted independence number}\label{S:weighted}

A set $U \subseteq V(G)$ is an \emph{independent set} of a graph $G$ if  the induced graph $G[U]$ has no edges. Let  $\I(G)$ denote the set of all the independent sets of $G$. The  \emph{independence number}  $\alpha(G)$ equals the size of 
a largest independent set of $G$.  
It is  well known that (see, for example \cite{Luczak1991, Frieze1990}) if $np \rightarrow \infty$ and $p<1-\eps$ {for a constant $\eps\in (0,1)$,} then  whp
\begin{equation}\label{chi-relation}
	\chi(\G(n,p))   = (1+o(1)) \frac{n}{ \alpha(\G(n,p))}.
\end{equation}
 That is,  for an asymptotically optimal colouring  of $\G(n,p)$,  almost all vertices  are covered with colour classes of approximately  equal size $\alpha(\G(n,p))$.

 One may think that, to approach the chromatic number  of inhomogeneous random graphs, one can also start with its independence number.
 In fact,   Dole\v{z}al et al.  \cite{DHM2017} studied  the clique  number in  inhomogeneous random graphs. Note that the clique number of a graph equals the independence number of  its complement.
   However,  we find little use of the results  of  \cite{DHM2017} in determining the chromatic number of  a random graph $\G\sim \SBM$ from  the stochastic block model.   
    Unlike   the homogeneous binomial random graph $\G(n,p)$,    some parts of the  random graph $\G\sim \SBM$
 will typically contain substantially larger independent sets than other parts of the graph so one  can not achieve an optimal colouring using colour classses of approximately same size.

 To take the inhomogeneity  of $\G \sim \M(\n,P)$ into account, we assign special weights to  subsets of vertices (depending on the edge probabilities)
and  introduce  a new parameter, called  \emph{weighted independence number}, which is the maximal weight of an independent set.  
Formally,  for a set $U \subseteq V(\G)$, define
\[
	h(U) = h(U,\n,P) := \frac{- \log (\Pr(U \in \I(\G)))}{|U|}.  
\]
Then, for a graph $G$ on vertex set  $V(G)=V(\G)$, let
\begin{equation}\label{def:alpha_w}
\begin{aligned}
	\alpha_h(G) = \alpha_h(G, \n,P) := \max_{ U \in\I(G),\, U\neq  \emptyset} h(U).
\end{aligned}
\end{equation}
It might be not obvious but nevertheless true that
 the weights $h(U)$  are designed in such a way that   all maximal independent  sets $U$  in the random graph $\G$ have similar weights whp.  This is 
  a natural generalisation of the idea of the balanced colouring of $\G(n,p)$ except we use the weight instead of the size of a colour class.

In this section,  we show, in particular  that,  under the assumptions of  Theorem \ref{Thm_sbm} and provided that {not all}  blocks  $B_i$ are very small, 
 the quantity $\alpha_{h}(\G)$ is concentrated around $(1-\sigma) \log \|\n\|$ whp; see Theorem \ref{T:independent}.
Moreover, we establish fast decreasing tail bounds for the probability of $\alpha_{h}(\G)$ being too large or too small; see 
{Lemmas \ref{l:upper_gen} and \ref{L:janson},} respectively. 
Lemma \ref{l:upper_gen} almost immediately leads to
the proof of Theorem~\ref{T:lower}.   Even though,
 Lemma \ref{L:janson} does not immediately give Theorem~\ref{T:upper}, 
 it will be the crucial instrument  for our construction of an optimal colouring of $\G$ in further sections.

Let $Q= Q(P)$ be defined by \eqref{Q_def}, where $P$ is the matrix of edge probabilities for $\G \sim \M(\n,P)$.
For simplicity,  everywhere in this section, let
\[
	 w(\cdot) \equiv w(\cdot,Q) \quad \text{and} \quad  w_*(\cdot) \equiv w_*(\cdot,Q);
\]
 see  \eqref{w_def}, \eqref{def:mu-star} for definitions.  Let  $q^*$ and  $\hat{q}(\cdot)$ be defined  according to \eqref{def:q-star}.
   In addition, we consider the vector-valued function $\b: 2^{V(\G)} \ \to \N^k$ that maps  $U \subseteq V(\G)$
  into $\b(U)$ defined by
\begin{equation*}
\b(U) = (b_1(U), \ldots, b_k(U))^T 
\quad\text{with}\quad  b_i(U):= |U\cap B_i| \quad\text{for }  i\in[k].
\end{equation*} 
Here, $B_i$ are the blocks of vertices in the stochastic block model $\M(\n,P)$.
Note that, for any $U \subseteq V(\G)$,  we have that $\|\b(U)\| = |U|$ and 
\begin{align}
 \b(U)^T Q\, \b(U) &= - 2\log  \left( 	\Pr(U \in  \I(\G)) \right)  + \sum_{i\in [k]} q_{ii}b_i(U)
 \label{eq:indMatrix1}
 \\
 		&\leq - 2\log  \left( 	\Pr(U \in  \I(\G)) \right)  + q^* |U|.
\label{eq:indMatrix2}
\end{align}




\subsection{Lower tail bound: proof of Theorem \ref{T:lower}}\label{S:lower}
{First, we estimate the probability of $\alpha_{h}(\G) $ to be large  for  a general random graph $\G$  with independent adjacencies.} 

\begin{lemma}\label{l:upper_gen}
	Let $\G$ be a random graph on $n$ vertices {where edges  appear independently of each other.} 
	Assume	  $s_t e^{t} \geq 6 n$ for some $t>0$,  where 
	\[
	  s_t: = \min\left\{|U| \st \emptyset \neq U\subseteq V(\G), \  \Pr(U\in \I(\G)) \leq e^{-t |U|}\right\}.
	  \] 
	   Then
	\[
	\Pr(\alpha_{h}(\G) \geq  t) \leq 2^{1-s_t}.
	\] 
\end{lemma}
\begin{proof}
	Let $X_s$ denote the number  of independent sets $U$ of size $s$  in $\G$ such that 
	\[
		\Pr(U \in \I(\G)) \leq e^{-t |U|}.
	\]
	By definition of $s_t$, we have that $X_s=0$ for any $s<s_t$.   If  $s\geq s_t$ then we bound 
	\begin{align*}	
	\Pr(X_s>0) &\leq \E X_s \leq 
	\sum_{U}   \Pr( U \in  \I(\G)) \\
	& 
	\leq \binom{n}{s} e^{- ts} \leq  \left( \frac{en  }{s e^{t}}  \right)^s \leq 2^{-s},
	\end{align*}
	where the sum is over all $U$ that contribute to $X_s$.
	Thus, we can bound 
	\begin{align*}
	\Pr(\alpha_{h}(\G) \geq   t)  \leq \sum_{s=s_t}^n \Pr(X_s>0)
	\leq  \sum_{s=s_t}^n  2^{-s} \leq 2^{1-s_t},
	\end{align*}
	which concludes the proof.
\end{proof}

{Next, applying Lemma \ref{l:upper_gen}  to $\G \sim \SBM$, we derive the required  probability bound for the event that the chromatic number $\chi(\G)$ is small.} 
\begin{proof}[Proof of Theorem \ref{T:lower}]
	Take $t:= \log (q^*\|\n\|)$.
	To apply Lemma \ref{l:upper_gen},  we need  to bound the quantity $s_t$ in Lemma \ref{l:upper_gen}. 
    If $U$ is such that $\Pr(U \in \I(\G)) \leq e^{-t |U|}$, then using \eqref{eq:indMatrix1}, we get that
   \[
   		t |U| \leq -  \log \left(\Pr(U \in \I(\G))\right) \leq \dfrac12 |U|^2 \max_{i,j\in {[k]}}q_{ij}. 
   \]
	Then, by the assumptions, we get that
	\begin{equation}\label{eq:st}
	s_t \geq \frac{2t}{ \max_{i,j\in [k]}q_{ij}}  = \frac{2 \log (q^*\|\n\|)} { \max_{i,j\in [k]}q_{ij}} \rightarrow \infty 
	\end{equation}
		and 
	\[
	s_t e^t \geq  \frac{2t e^t}
	{\max_{i,j\in [k]}q_{ij}} = \frac{2 q^*   \|\n\|   \log (q^*\|\n\|)} { \max_{i,j\in [k]}q_{ij}}\gg \|\n\|.
	\]
	Applying Lemma \ref{l:upper_gen}, we  find that
	\begin{equation}\label{eq:alpha}
		\Pr(\alpha_h(\G) \geq \log (q^*\|\n\|))  \leq  2^{1-s_t}.
	\end{equation}
 Next, using {the corner maximiser property} in Theorem \ref{l:Qmatrix}(b) and \eqref{eq:indMatrix2}, we   find that, for any $U \in \I(\G)$,
	\begin{equation}\label{eq:phi-w}
		 w(\b(U)) = \max_{\emptyset \neq W \subseteq U} \dfrac{\b(W)^T Q \b(W)}{|W|} 
		 \leq   \max_{\emptyset \neq W \subseteq U}  \left(-\dfrac{2\Pr(W \in \I(\G))}{|W|} + q^*\right)
		 \leq 2 \alpha_h(\G) + q^*.
	\end{equation}
	In the above, we also used that if $W\subseteq U \in \I(\G)$ then $W \in \I(\G)$.
{	Recall that 
	$\log (q^*\|\n\|) = \Theta(\log \|\n\|)$ by \eqref{ass_sbm} and
	 $q^* \leq \max_{i,j \in [k]} q_{ij} \ll \log \|\n\|$ by \eqref{ass_1}. Thus, if   $\alpha_h(\G) \leq \log (q^*\|\n\|)$ then, by \eqref{eq:phi-w}, we have that} 
	\[
			\max_{U \in \I(\G)}  w(\b(U))  \leq  {2}\log (q^*\|\n\|)) + q^* =    (2+o(1)) \log (q^*\|\n\|)).
	\]
	
	Let $\{U_i\}_{i =1, \ldots, \chi(\G)}$ be the partition of $V(\G)$
	into colour classes of any  optimal colouring of  $\G$. Cosidering the system consisting of vectors $\b(U_i)$ for $i =1, \ldots, \chi(\G)$,   and recalling definition \eqref{def:mu-star}, we find  that
	\[
	w_*(\n) \leq \sum_{i =1}^{\chi(\G)} w(\b(U_i)).
	\]
	 We conclude that,   with probability at least $1-  2^{1-s_t}$,  
	\[
	\chi(\G) \geq \frac{\sum_{i =1}^{\chi(\G)} w(\b(U_i))}{\max_i w(\b(U_i))}
	\geq   \frac{w_*(\n)}{(2+o(1)) \log (q^*\|\n\|)} \geq (1-\eps)  \frac{w_*(\n)}{2 \log (q^*\|\n\|)}.
	\]
	Using \eqref{eq:st}, we get that
	\[
		2^{1-s_t}
	= \exp\left( - \Omega\left( \frac{ \log(q^*\|\n\|)}{\max_{i,j\in [k]}q_{ij}} \right)\right).
      \]
      This completes the proof.
\end{proof}


\subsection{Existence of heavy independent sets}\label{S:existence_he}

We  consider a special class of sets   distributed between the blocks $B_1, \ldots ,B_k$  proportionally to its sizes (up to rounding). For a vector $\xvec = (x_1,\ldots,x_k)^T \in \Reals^k$, denote 
\[ \lfloor \xvec\rfloor := (\lfloor x_1\rfloor, \ldots, \lfloor x_k\rfloor)^T
\qquad \text{and} \qquad x_*:=\min_{i\in[k]}x_i.
\]
For a positive real $\nu$,  let  
$\I_{ \nu}(\G)$ denote the family  of  
independent sets  $U\subseteq \I(\G)$  such that 
$\b(U) = \lfloor \nu \n \rfloor$. 

\begin{lemma} \label{L:janson}
	Let $\G \sim \SBM$, where $P = (p_{ij})_{i,j\in [k]}$ is such that $p_{ij} = p_{ji}$ and  $0\leq p_{ij}<1$ for all $i,j\in [k]$.
	Let $Q= Q(P)$ be as in \eqref{Q_def} and $\sigma \in [0,\sigma_0)$ for some fixed $0<\sigma_0<\frac12$.
	Assume that   $
	\|\n\| \rightarrow \infty$, 
	$ w (\n) \geq \|\n\|^{1-\sigma}, 
	$
	\begin{equation}
	 n_* = \|\n\|^{1+o(1)}, \qquad  n_* \gg \frac{w(\n)}{\log \|\n\|}, \label{eq:n*bound}
	\end{equation}
	where  $n_*:=\min_{i\in[k]}n_i$.
	Then, {there exists $\nu \in \pReals$ such that} 
	$
	\nu = (2 + o(1)) \dfrac{\log (w(\n))}{w (\n)}
	$
 and
	\begin{equation*} 
	\Pr \left(\I_{\nu}(\G) = \emptyset \right) \leq
	\exp\left(   -  \|\n\|^{ 2-4\sigma +o(1)}  \right). 
	\end{equation*}
\end{lemma}
\begin{proof}
	Since $n_* \gg \dfrac{w(\n)}{\log \|\n\|}$, we can find some  $r(\n)$ such that 
	\begin{equation}\label{def_r}
	\dfrac{w(\n)}{n_*}    \ll r(\n) \ll\log \|\n\|.
	\end{equation}
	For example, one can take $r(\n) := \left(\dfrac{w(\n)}{n_*}   \log \|\n\|\right)^{\frac12}$.
	Define
	\begin{equation}\label{def:nu}
	\nu := \frac{2}{w (\n)} \bigg(   \log (w (\n)) - 2 \log \log(w (\n)) -   \log \left( \dfrac{ \|\n\|}{n_*}
	\right)  -  r(\n)\bigg). 
	\end{equation}
	Note that the assumptions imply that 
	\begin{equation}\label{eq:nuh}
	\log (w (\n)) \geq \frac12 \log \|\n\|, \qquad \nu = (2 + o(1)) \frac{\log (w(\n))}{w (\n)}.
	\end{equation}
	Let  $\l =(\ell_1,\ldots, \ell_k) = \lfloor \nu \n \rfloor$. That is, we have $\l = \b(U)$  
	for all $U \in \I_{\nu}(\G)$.  
	Using the assumptions,  we get, for all $i \in [k]$, 
	\begin{equation}\label{eq:s+l}
	\ell_i = (2 + o(1))  \frac{n_i \log (w(\n))}{w (\n)}
	\geq    (2 +o(1))  \frac{n_* \log (w(\n))}{w (\n)} 
	\gg 1.
	\end{equation}
	Observe that the number of ways to pick 
	the set $U \in V(\G)$ such that $\b(U) =\l$ 
	equals
	$
	\prod_{i=1}^k\binom{n_i}{\ell_i}. 
	$
	Then, using \eqref{eq:indMatrix1} {that  relates $\Pr(U \in \I_{\nu}(\G))$
	and $e^{- \frac{\b(U)^T Q\, \b(U)}{2}}=e^{- \frac{\l^T Q\, \l}{2}}$},  we get that 
	\begin{align*}
	\E |\I_{\nu}(\G)| 
	&=  \Pr(U \in \I_{\nu}(\G))   \prod_{i=1}^k\binom{n_i}{\ell_i} 
	= e^{- \frac{\l^T Q \l}{2}}   \prod_{i=1}^k\binom{n_i}{\ell_i}     (1-p_{ii})^{-\ell_i/2}    \\
	&\geq   e^{- \frac{\l^T Q \l}{2}}   \prod_{i=1}^k\left(\frac{n_i}{\ell_i}\right)^{\ell_i}    
	\geq  e^{- \frac{\l^T Q \l}{2}} \nu^{-\|\l\|} .
	\end{align*}
	Using  {the scaling property in Theorem} \ref{l:Qmatrix}(a) and {the definition  \eqref{def:nu} of $\nu$},  we get 
	\begin{equation}\label{l-bound}
	\frac{\l^T \, Q\, \l}{2\|\l\|}
	= \nu \frac{(\nu^{-1}\l)^T  \, Q\, (\nu^{-1}\l)}{2\|\nu^{-1}\l\|} 
	\leq  \frac{\nu \, w(\n)}{2}  =
	\log   \left( \frac{w(\n) n_*}{\log^2 (w(\n)) \|\n\|}\right) - r(\n). 
	\end{equation}
	From \eqref{eq:s+l}, we also get that 
	\[
	\|\l\| = (2 + o(1))  \frac{\|\n\| \log(w(\n))}{ w(\n)}.
	\]
	Using \eqref{def_r}, {\eqref{eq:nuh}, \eqref{l-bound}},   the obvious inequality $n_* \leq \|\n\|$, and $w(\n) \geq  \|\n\|^{1-\sigma}$, we get that 
	\begin{equation}\label{1-moment}
	\begin{aligned}
	\E |\I_{\nu}(\G)|  &\geq  e^{- \frac{\l^T Q \l}{2}} \nu^{-\|\l\|}  
	\geq 
	\left(\left(\dfrac12 + o(1)\right) \log (w(\n)) \frac{   \|\n\|}{  {n_*}}  e^{r(\n)})\right)^{\|\l\|}
	\\
	&\gg e^{\|\l\|r(\n)}  =  \exp \left(  \omega\left(  \frac{\|\l\| w(\n)}{n_*}\right) \right)
	= \|\n\|^{\omega(1)}.
	\end{aligned}
	\end{equation}

	Next, let 
	\[ 
	\Delta := \sum_{\substack{|U \cap W| \geq 2}}
	\Pr \left(U \in \I_{ \nu}(\G) \text{ and } W \in \I_{ \nu}(\G)\right),
	\] 
	where the sum is over 
	all possible ordered pairs $(U, W)$ of subsets of $V(\G)$
	such that $|U \cap W| \geq 2$.  
	Note that if $|U \cap W| \leq  1$
	then the events $\{U \in \I_{\nu }(\G)\}$ and  
	$\{W \in \I_{\nu}(\G)\}$ are independent.
	By Janson's inequality, see \cite[Theorem 1]{Janson1990}, we have 
	\begin{equation}\label{Janson}
	\Pr \left(\I_{\nu}(\G) = \emptyset \right)
	\leq \exp \left(- \frac{(\E |\I_{\nu}(\G)|)^2}{ 2 \E |\I_{\nu}(\G)| + 2\Delta}\right).
	\end{equation}
	We have already established a lower bound for $\E |\I_{\nu}(\G)|$ in \eqref{1-moment}.
	Thus, it remains to bound  $\frac{\Delta }{(\E |\I_{\nu}(\G)|)^2}$ from the above.
	Using \eqref{eq:indMatrix1}, we find that
	\begin{align*}
	\frac{\Delta }{(\E |\I_{\nu}(\G)|)^2}=  \sum_{\m} 
	e^{\frac{\m^T   Q  \m}{2}} 
	\prod_{i=1}^k  \frac{\binom{\ell_i}{m_i}\binom{n_i-\ell_i}{\ell_i-m_i}}
	{\binom{n_i}{\ell_i}}   (1-p_{ii})^{m_i/2},
	\end{align*}
	where the sums are over $\m =(m_1,\ldots,m_k)^T \in    \Naturals^k$
	with  $\|\m\|\geq 2$ and $\m \preceq \l$.
	Observe that
	\begin{align*}
	\frac{\binom{\ell_i}{m_i}\binom{n_i-\ell_i}{\ell_i-m_i}}
	{\binom{n_i}{\ell_i}} 
	&= \frac{((\ell_i)_{m_i})^2 (n_i-\ell_i)_{\ell_i-m_i}}{m_i! (n_i)_{\ell_i}}	 
	\leq \frac{((\ell_i)_{m_i})^2}{ m_i! (n_i)_{m_i}} 
	\leq  \frac{1}{m_i!} \left(\frac{\ell_i^2}{n_i}\right)^{m_i}
	\\ &=  \frac{\left((1+o(1)) \nu^2 n_i \right)^{m_i} }{m_i!}  
	\leq  \frac{ 1}{m_i!}  \left(5   n_i\left(\frac{\log(w(\n))}{w(\n)}\right)^2 \right)^{m_i}.
	\end{align*}
	Denote
	\[
	\theta_m := \max_{\|\m\| = m} \frac{\m^TQ\m}{2\|\m\|},
	\]
	where the maximum is over $\m \in  \Naturals^k$ with $\|\m\|=m$ and $\m \preceq \l$.
	Then, we obtain
	\begin{equation}\label{Big_sum}
	\begin{aligned}
	\frac{\Delta }{(\E |\I_{\nu}(\G)|)^2} &\leq \sum_{m=2}^{\|\l\|}
	\left(5 \left(\frac{\log(w(\n))}{w(\n)}\right)^2 e^{\theta_m}\right)^{m_i} \prod_{i=1}^k \frac{n_i^{m_i}}{m_i!}
	\\ 
	&=\sum_{m=2}^{\|\l\|}
	\frac{1}{m!} \left(5 \|\n\| \left(\frac{\log(w(\n))}{w(\n)}\right)^2 e^{\theta_m}\right)^{m}.
	\end{aligned}             
	\end{equation}
	
	There are two ways we can estimate the quantity $\theta_m$. First, repeating the arguments of  \eqref{l-bound}
	with $\l$ replaced by any $\m \preceq \l$, we find that
	\begin{equation}\label{m-bound1}
	\theta_m
	\leq  \frac{\nu \, w(\n)}{2} =  
	\log   \left( \frac{w(\n) n_*}{\log^2 (w(\n)) \|\n\|}\right) -  r(\n).
	\end{equation} 
	Second,   observing 
	$ \frac{n_*}{\|\m\|} \m   \preceq \n$ and
	using {the monotonicity property in Theorem  \ref{l:Qmatrix}(a)}, we get
	\[
	\frac{\m^T \, Q\, \m}{2\|\m\|}  \leq   \frac{\|\m\| \,  w\left(    \frac{n_*}{\|\m\|} \m  \right)}{2 \ n_*  }
	\leq  \frac{\|\m\| w(\n)}{2 \, n_*   }. 
	\]
	Thus, we get 
	\begin{equation}\label{m-bound2}
	\theta_m \leq \frac{m\, w(\n)}{2 n_* }, 
	\end{equation} 
	which is better than \eqref{m-bound1} for small $m$.

	Using \eqref{def_r},  we can find  $m_0 \in \Naturals$  such that  
	\begin{equation}\label{m0-inequalities}
	1\ll \frac{\log \|\n\|}{r(\n)}
	+ \frac{n_*}{w(\n)} 
	\ll m_0 \ll \frac{ n_* \log \|\n\|    }{w(\n)}.
	\end{equation}
	Using the inequality $m! \geq m^m e^{-m}$ and the bound 
	{$e^{\theta_m} \leq  \dfrac{w(\n) n_*}{\log^2 (w(\n)) \|\n\|}  e^{-r(\n)}$} implied by \eqref{m-bound1}, we find that
	\begin{equation}\label{Big_sum1}
	\begin{aligned}
	\sum_{m  = m_0 }^{\|\l\|}  \frac{1}{m!} 
	\left(5 \|\n\| \left(\frac{\log(w(\n))}{w(\n)}\right)^2 e^{\theta_m}\right)^{m}
	&\leq
	\sum_{m=m_0 }^{\|\l\|} \left( \frac{5 e^{1-r(\n)} n_*}{m \, w(\n) }  \right)^{m} 
	\\
	\ll \sum_{m=m_0 }^{\|\l\|} e^{-  m r(\n)} 
	&\leq \|\l\| e^{-\omega(\log  \|\n\|)}  =  \|\n\|^{-\omega(1)},
	\end{aligned}
\end{equation}
{where the last two inequalities  used the lower bounds of \eqref{m0-inequalities}: first $m\geq m_0 \gg  \frac{n_*}{w(\n)}$ and then  $m \geq m_0 \gg \frac{\log \|\n\|}{r(\n)}$.}  
	we have  $\theta_m  \ll \log \|\n\|$ by  \eqref{m-bound2}.
	Recalling our assumptions 
	that $w(\n)\geq \|\n\|^{ 1-\sigma }$ and {$\|\n\|^{1+o(1)} =n_* \gg \dfrac{w(\n)}{\log \|\n\|}$},  we find that the following sum is dominated by the first term:
	\begin{equation}\label{Big_sum2}
	\sum_{m=2}^{m_0-1} \frac{1}{m!}\left(\frac{5 (\log (w(\n)))^2  \|\n\|e^{\theta_m}}{(w(\n))^2}\right)^m
	= 
	\left(\dfrac12+o(1)\right)
	\left(\frac{\|\n\|  e^{o(\log \|\n\|)}}{(w(\n))^2} \right)^2
	\leq \|\n\|^{ 4\sigma -2   +o(1)}.
	\end{equation}
  {Putting \eqref{Big_sum1} and \eqref{Big_sum2} in  \eqref{Big_sum},
	we obtain that
	\begin{align*}
	\frac{\Delta}{(\E |\I_{\nu}(\G)|)^2}   
 \leq
	\|\n\|^{2 -4\sigma  +o(1)} + \|\n\|^{-\omega(1)} =
	\|\n\|^{2-4\sigma    +o(1)}. 
	\end{align*}
Recalling  from \eqref{1-moment} that $\E |\I_{\nu}(\G)| = \|\n\|^{\omega(1)}$, we conclude that 
	\begin{align*}
	\frac{\Delta +\E |\I_{\nu}(\G)|}{(\E |\I_{\nu}(\G)|)^2}   
	 \leq
	\|\n\|^{2-4\sigma    +o(1)}. 
\end{align*}
	Applying \eqref{Janson}, we complete the proof. }
\end{proof}

\subsection{Concentration of the weighted independence  number}

The estimates of  Sections \ref{S:lower} and  \ref{S:existence_he} lead  to the following result.

\begin{thm}\label{T:independent}
Let $\G \sim \SBM$, where $P = (p_{ij})_{i,j\in [k]}$ is such that $p_{ij} = p_{ji}$ and $0 \leq p_{ij}<1$ for all $i,j \in [k]$.  Let $Q= Q(P)$ be as in \eqref{Q_def}. Let $\sigma \in [0, \sigma_0]$  for some fixed  $0<\sigma_0 < \frac12$.  
	Assume that   \eqref{ass_sbm}, \eqref{ass_1}  hold
 and 
 \[
 	 n_* = \|\n\|^{1+o(1)}, \qquad  n_* \gg \frac{w(\n)}{\log \|\n\|},
 \]
 where  $n_*:=\min_{i\in[k]}n_i$.
		Then,   whp 
	\[
	\alpha_{h}(\G)  = (1-\sigma + o(1))  \log \|\n\|.
	\]
\end{thm}
\begin{proof}
 All the assumptions of Theorem \ref{T:lower} also present in this theorem,  so we can use the formulas and arguments given in its proof.
    Using  \eqref{eq:alpha} and {the assumption $ q^* =    \|\n\|^{-\sigma+o(1)}$ by \eqref{ass_sbm}}, we find that whp 
    \[
    	\alpha_{h}(\G) \leq  \log (q^*\|\n\|) = (1-\sigma+o(1)) \log \|\n\|.
    	    \]
  Next, using {Theorem  \ref{l:Qmatrix}(d) and the assumptions $ \hat{q}(\n), q^* =    \|\n\|^{-\sigma+o(1)}$ by \eqref{ass_sbm}}, we have that 
  \[
  	w(\n) \geq w_*(\n) \geq \frac{(\hat{q}(\n))^2}{kq^*} \|\n\|  = \|\n\|^{1-\sigma+o(1)}. 
  \]
  Thus, all assumptions of Lemma \ref{L:janson} hold.  
   Applying Lemma \ref{L:janson}, we find that  whp $\I_\nu(\G)   \neq\emptyset$ for some $\nu = (2+o(1)) \frac{\log (w(\n))}{w(\n)}$.
   If $U \in \I_\nu(\G) $ then for all $i \in [k]$ 
   \[
   	b_i(U) = \lfloor \nu n_i\rfloor \geq  \left(\nu -  \dfrac{1}{n_*}\right) n_i.
   \]
  Combining {the above, the monotonicity property in Theorem}  \ref{l:Qmatrix}(a), and  \eqref{ass_sbm}, we get that 
   \[
   	w({\b(U)}) \geq (\nu -  \dfrac{1}{n_*}) w(\n) = (2+o(1)) \log (w(\n)) \geq (2-2\sigma +o(1)) \log \|\n\|.
   	 \]
 Thus,  using \eqref{eq:phi-w} and {the arguments below 
 \eqref{eq:phi-w} showing that $q^*\ll \log \|\n\|$,} we find that, whp 
 \[
       \alpha_h(\G)  \geq  \tfrac12 \left(w({\b(U)})  -  q^*\right) = (1-\sigma  + o(1))\log \|\n\|.
 \] 
    	 This completes the proof.
\end{proof}

 We note that the proof of our main result, Theorem \ref{Thm_sbm},  does not rely on Theorem \ref{T:independent}, but the study of  the distribution of the parameter $\alpha_h(\G)$ is of independent interest. 
 Observe that definition \eqref{def:alpha_w}   extends to any random graph model.  We believe that Theorem~\ref{T:independent} carries over as well. In particular, we conjecture the following.
 \begin{conj}\label{conj:independent}
   Let $\G = \G(n)$ be a random graph on vertex set $[n]$ where edges $ij$ appear independently of each other with probabilities $p_{ij} = p_{ij}(n) \in (0,1)$.   Assume that there exist
   $q = q(n)$  and constants $c_1,c_2>0$ such that
    \[
    	 q n \rightarrow \infty, \qquad q \ll  \log n,
    \]
     as $n \rightarrow \infty$, and,     for all edges $ij$, 
   \[
   	e^{-c_1 q} \leq 1- p_{ij} \leq e^{-c_2 q}.
   \]
   Then, whp
   \[
   	\alpha_h(\G) = (1+o(1)) \log(qn).
   \]
 \end{conj}

  In fact, 
  proceeding from 
  Lemma \ref{l:upper_gen} similarly to \eqref{eq:alpha}, 
  one can derive that $\alpha_h(\G) \leq \log (qn)$ whp
  under the assumptions of Conjecture \ref{conj:independent}. 
  However, proving the counterpart would require significant modifications of the arguments given in Section \ref{S:existence_he}.


\section{Crude upper bound}\label{S:upperbound}
 
In this section, we establish a crude upper bound on $\chi(\G)$, where $\G \sim \M(\n,P)$
based   on a simple idea of colouring each block separately.  To do so, we only need   the results for the classical binomial random graph $\G(n,p)$. 

\begin{lemma}\label{L:crude_upper}
	Let $\sigma \in [0,\sigma_0]$  for some fixed $0<\sigma_0<\frac14$ and    $p=p(n)\in (0,1)$ is such that  
	\[
	\log n \gg
	q := \log  \dfrac{1}{1-p} \geq n^{{-\sigma}}.
	\]   
	Then, for any $\eps>0$ and any $s = n^{1+o(1)}$, 
	{with probability at least
	$
	1 - \exp\left(- n^{2-4\sigma + o(1)}\right)
	$,
	there is a colouring of $\G(n,p)$ with at least } 
	$
	n- s 
	$
	vertices using at most 
	$ \left(1 + \eps\right) \frac{ qn  }{2\log \left(qn \right)}$ colours.
\end{lemma}

\begin{proof} 
	This  argument is well known for a constant $p \in (0,1)$; see for example, \cite[Section 7.4]{FK2015}. For the sake of completeness, we repeat it here and  check that it extends  to $p=p(n)$
	satisfying the assumptions of Lemma \ref{L:crude_upper}.
	
	We will apply Lemma \ref{L:janson}  to subgraphs  $\G(n',p) $
	of $\G(n,p) $ with $n'\geq s$, by setting    $k=1$, 
	$\n = (n')$, and $P = (p)$. Then, by definition, we have  $n_* =\|\n\| = n'$ and $w(\n) =  qn'$ so
	all assumptions of Lemma \ref{L:janson} hold.  
	Using   Lemma \ref{L:janson}, 
	we  show that the probability 
	that there is a subgraph in $\G(n,p)$ with at least $s$ vertices 
	without an independent set of size $(2-\eps) \frac{\log (qn)}{q}$  is at most 
	\[
	2^n  \exp(- s^{2-4\sigma + o(1)}) = \exp\left(- n^{2-4\sigma + o(1)}\right).
	\]
	Thus, we can keep colouring such independent sets and {deleting} them from the graph until we are left with {fewer} than $s $ vertices. 
	The number of colours used in this  process is bounded above by 
	\[
	\frac{n}{(2-\eps) \frac{\log (qn)}{q}} \leq \left(\frac12 + \eps\right) \frac{ qn  }{\log \left(qn \right)}.
	\]
	Note that, in the above, we can assume that $\eps<1$   since the statement of the lemma becomes stronger. Then,   the inequality $ \dfrac{1}{2-\eps} \leq  \dfrac{1+\eps}{2}$  holds.
\end{proof}

To colour the remaining vertices, we use the following lemma. 

\begin{lemma}\label{L:density}
	Let $\G =\G(n,p)$, where   {$p=p(n)\in  (0,1)$} is such that $pn \rightarrow \infty$ as $n\rightarrow \infty$.  Then, for any positive integer $s \geq p^{-1}$, we have   
		\[
	\Pr \Big( \exists\ W \subseteq V(\G) \st |W| = s \text{ and  $\chi( \G[W]) \geq  p s \log  n +1$ } \Big)
	\leq \exp\left(-\omega(p^2s^2 \log^2 n)\right).
	\]
\end{lemma}
\begin{proof}
	First,  for any postive integer $u \leq s$,  
	we  estimate the probability of the event that  the minimal degree of $\G'=\G(u,p)$  
	is at least {$ps \log  n$}.   This event implies that $\G'$
	has at least $\frac12 psu \log  n$ edges. Let $N_u:= \binom{u}{2}$. 
	Since the distribution of the number of edges in $\G'$ is 
	$\operatorname{Bin}(N_u,p)$, we find that 
		\begin{align*}
	\Pr\left(\G' \text{ has at least $\dfrac12 psu \log  n$ edges} \right)&=
	\sum_{i \geq  \frac12 psu \log  n} \binom{N_u}{i} p^i (1-p)^{N_u-i}
	\\
	\leq  \sum_{i \geq \frac12 psu \log  n} \left(  \frac{e p N_u}{i}\right)^i
	&\leq 2 \left(  \frac{e }{ \log n}\right)^{ \frac{1}{2} psu\log n} =
	n^{- \omega(psu)}.
	\end{align*}
	To derive the last inequality in the above, we observe  that  
	\[ 
	\frac{  p N_u}{i} \leq \frac{p u^2 }{psu \log n} \leq \frac{1}{ \log n}.
	\]
	Note also that if $u \leq ps\log n$ then $\G'$ has less than   $\frac12 psu\log n$ edges with probability 1. 
	
	Using the union bound over all choices 
	for $W \subseteq V(\G)$ with $|W|=s$,    for	$u$ such that  $ps \log n< u \leq s$ ,	and for $U \subseteq W$ with $|U|=u$, we  get that
	\begin{align*}
	\Pr  \Big(  \exists  W \subseteq V(\G) &\st |W| = s   \text{ and }
	\max_{U \subseteq W} \delta_{\G} ({U}) \geq ps \log n\Big)
	\\ 
	&\leq
	2 \binom{n}{s} \sum_{u > ps\log n} \binom{s}{u}  n^{-\omega(psu)}
	\\ &\leq 2 \left(\frac{en}{s}\right)^s   \sum_{u > ps\log n} \left(\frac{es}{u}n^{-\omega(ps)}\right)^u
	=   \exp\left(-\omega(p^2s^2 \log^2 n)\right).
	\end{align*}
	{Combining this and  the upper bound  \eqref{chi:upper} on the chromatic number  in terms of the minimal degree of subgraphs, we  complete the proof.} 
\end{proof}

Combining Lemma \ref{L:crude_upper} and  Lemma \ref{L:density}, we get the following result
for  $\G \sim \SBM$.  
 {This result will be important in the proof of Theorem \ref{T:upper} to show that the number of colours required for the remaining vertices (given by a  set $U \subseteq V(\G)$) after a certain "optimal"  colouring process is negligible.}

\begin{thm}\label{T:uppercrude}
	Let $\G \sim \SBM$,  where $P = (p_{ij})_{i,j\in [k]}$ is such that $p_{ij}=p_{ji}$ and $0\leq p_{ij}<1$ for all $i,j\in [k]$.  
	Let $q^*$ and $\hat{q}$ be defined in \eqref{def:q-star}.  
	Let $\sigma \in [0,\sigma_0]$  for some fixed $0<\sigma_0<\frac14$.  Assume that
	$n$ is such that 
	\[
	n \geq \|\n\|^{1+o(1)},\qquad 
	k   = n^{o(1)}, \qquad  \log n \gg q^* \geq {n^{-\sigma}.}
	\]
	Let  $\uvec= (u_1,\ldots,u_k)^T \in \pReals^k$ be  such that  
	$
		\hat{q}(\uvec)  =  q^*n^{o(1)}.
	$
	Then,    for any $\eps>0$,
	\[
	\Pr\left(\max_{\b(U)\leq \uvec}\chi(\G[U]) >  \left(1 + \eps \right) 
	\frac{  \hat{q}(\uvec)  \|\uvec\| }
	{2 \log \left( q^* n\right)} \right) \leq \exp \left(- n^{2-4\sigma +o(1)}\right),
	\]
	where the maximum is over  all subsets $U \subseteq V(\G) $  such that $|U\cap B_i| \leq u_i$ for all $i\in [k]$.
\end{thm}    

\begin{proof}
	For each $i\in[k]$,  we let  $\G_i   = \G(n_i, p_{ii})$  denote the induced subgraph of $\G[B_i]$. 
	Let   
	\[
	n_i' := n_i+s, \qquad \text{where }  s := \left\lfloor\frac{ \hat{q}(\uvec)  \|\uvec\|}{k  q^* \log^3   \|\n\|}\right\rfloor.
	\]    
	Define $p_{ii}' \in (0,1)$ to be such that 
	\[
	q_{ii}':= \log \frac{1}{1-p_{ii}'}=  q_{ii} +  q_0,  
	\qquad \text{where }
	q_0:=\frac{ \hat{q}(\uvec)    }{k   \log \|\n\|}.
	\]
	Since $\log n\gg  q^*   \geq {n^{-\sigma}}$ and $\hat{q}(\uvec) = q^*n^{o(1)}$, we get that 
	\begin{equation}\label{eq:q'}
	\log n \gg 	q_{ii} +o(q^*) \geq q_{ii}'  \geq  q_0\geq 
	n^{-\sigma+o(1)}.  
	\end{equation}
	{By adding  $s$ dummy vertices to each block $B_i$ and introducing some rejection probability,  for each $i\in[k]$ there is a coupling $(\G_i',\G_i)$ such that $\G_i' = \G(n_i', p_{ii}')$  and $\G_i$ is a subgraph of $\G_i'$ with probability $1$.}

	Consider any $U_i \subseteq V(\G_i)$ such that  $|U_i|  \leq u_i$ and let $U_i'$ consist of the 
	union of $U_i$ and    $s$ dummy vertices  of $\G_i'$. Note that, by assumptions,  
	\[
	n^{1+o(1)} \leq  s\leq  u_i':=|U_i'| \leq u_i +  s.
	\]
		Using   \eqref{eq:q'},  we find that  $sq_0 =q^* \|\uvec\|  n^{o(1)} $ and 
	\begin{equation}\label{s:bounds}
	 \begin{aligned}(u_i +s) (q_{ii} +q_0) &=  u_i q_{ii}  +   o\left( \frac{\hat q(\uvec) \|\uvec\|  n^{o(1)}}{k} \right),
	 	\\(u_i +s) (q_{ii} +q_0) 
	&\leq (1+o(1)) q^*\|\uvec\|.
	\end{aligned} 
	\end{equation}

	Since $  q^*  \geq {n^{-\sigma}}$,  we get that  
	\[
	\log (q_{ii}'u_i')  =
	(1+o(1)) \log \left(q^*n\right).
	\]
	Using  \eqref{eq:q'}, we find that the assumptions of Lemma \ref{L:crude_upper} hold
	for $\G_i[U_i']= \G(u_i', p_{ii}')$.
	Applying  Lemma \ref{L:crude_upper} with $\eps'=\dfrac{\eps}{2}$, we find that
	there is a colouring of    $u_i' - s = u_i$ vertices of  $\G_i'$   using at most 
	\[
	\left(1+ \eps'\right) \frac{q_{ii}' u_i'}{ 2\log (q_{ii}'u_i')} 
	\leq  \left(1 + \dfrac{\eps}{2}+ o(1)\right)   \frac{q_{ii}'  u_i'}{ 2\log \left(q^*n\right)} 
	\]
	colours with probability at least 
	\[
	1 - \exp\left(- (u_i')^{2-4\sigma+ o(1)}\right) \geq
	1  - \exp\left(- n^{2-4\sigma + o(1)}\right). 
	\]
	Recalling that  $2-4\sigma>1$  and applying the union bound,  we get that
	the probability that there are some $i\in[k]$ and  $U_i  \subseteq V(\G_i)$ with $|U_i| \leq u_i$ for which such colouring does not exist is bounded above by
	\[
	\sum_{i=1}^k 2^{n_i + s} \exp\left(- n^{2-4\sigma+ o(1)}\right)
	=\exp\left(- n^{2-4\sigma + o(1)} \right).
	\]
	
	Next, we show that only a small number of colours is needed to colour the remaining $s$ vertices from each $V(\G_i')$. 
	Applying Lemma \ref{L:density}, we get that  any subset 
	$W\subseteq V(\G_i')$ with $s  \geq  n^{1+o(1)} \geq (p_{ii}')^{-1}$ vertices can be coloured using  at most
	\[
	p_{ii}' s \log n_i'+1  \leq q^* s \log n_i' +1 \ll \frac{ \hat{q}(\uvec) \|\uvec\|}{k \log (q^*n)}
	\]
	colours   with probability at least 
	\[
	1-  \exp\left(-\omega((p_{ii}')^2s^2 \log^2 n_i')\right)  
	\geq  1 - \exp\left(- n^{2-2\sigma+o(1)}\right).
	\]
	The last inequality is clear for $p_{ii}'\geq \frac12$. 
	For $p_{ii}'<\frac12$,  one can use \eqref{eq:q'} together with the inequality 
	$
	p_{ii}' \geq \frac{q_{ii}'}{ 2\log 2},
	$
	which follows from the fact that
	$t^{-1}\log \frac{1}{1-t}$ is monotonically increasing for $t\in (0,1)$. 
	Applying the union bound,   
	we can complete the colouring of all sets $U_i'$ for $i\in[k]$
	using at most 
	\[
	\sum_{i\in [k]} (p_{ii}' s \log n_i'+1 ) \ll   \frac{  \hat{q}(\uvec) \|\uvec\|}{\log (q^*n)}
	\] 
	colours with probability at least  	
	\[
	1 - k \exp\left(- n^{2-2\sigma+o(1)}\right) 
	\geq 1 - \exp\left(-n^{2-4\sigma+o(1)}\right).
	\]
	
	Now, consider any $U\subseteq V(\G)$  such that $|U\cap B_i| \leq u_i$ for all $i\in [k]$.  
	Combining the above bounds and  using the inequality in the second line of \eqref{s:bounds}, 
	we  get that  $U$  can be coloured with  at most 
	\begin{align*}
	\sum_{i\in [k]} &\left(1+\frac{ \eps}{2}+ o(1)\right)   \frac{q_{ii}' u_i'}{2 \log (q^*n)} 
	+ \sum_{i\in [k]} (p_{ii}'s \log n_i'+1 ) 
	\\
	&=  \left(1+ \frac{\eps}{2}+ o(1)\right)   \frac{  \sum_{i \in k} \left( u_i q_{ii} + o\left(\frac{ \hat{q}(\uvec) \|\uvec\|}{k}\right) \right)}{2\log (q^*n)}
	+  o\left(\frac{ \hat{q}(\uvec) \|\uvec\|}{\log (q^*n)}\right)
	\\     &\leq  \left( 1 + \eps \right) \frac{  \hat{q}(\uvec) \|\uvec\| }{2\log (q^*n)}
	\end{align*}
	colours with probability at least $ 1 - \exp\left(- n^{2-4\sigma+o(1)}\right)$.
\end{proof}


\section{Optimal colouring: proof of Theorem \ref{T:upper}}\label{S:optimal} 
In this section we prove Theorem \ref{T:upper}. First,  applying Lemma \ref{L:janson} multiple times, we  find 
 there are approximately $\frac{w(\n)}{2 \log (q^* \|\n\|)}$  independent sets covering almost all vertices of $\G \sim \M(\n,P)$. Then, we use Theorem \ref{T:uppercrude}  to estimate the number of colours for the remaining vertices, proving that 
 \[
 	\chi(\G) \leq (1+o(1)) \frac{w(\n)}{2 \log (q^* \|\n\|)} + {O \left( \frac{k \hat{q}(\n) q^* \|\n\|}{\log^2 (q^* \|\n\|)}\right)}
 \]
 with probability sufficiently close to $1$.  Finally, we obtain Theorem \ref{T:upper} by  applying this upper bound to each random graph 
 {corresponding}  to an optimal system of  $k$ vectors  $(\xvec^{(t)})_{t\in [k]}$ from $\pReals^k$  such that 
 \[
 	\n =  \sum\nolimits_{t \in [k]}\xvec^{(t)} \qquad \text{and} \qquad 	w_*(\n) = \sum\nolimits_{t \in [k]} w(\xvec^{(t)}). 
 \]
Everywhere in this section, we use notations $w(\cdot)$ and  $w_*(\cdot)$ in place of   $w(\cdot,Q)$ and  $w_*(\cdot,Q)$, where $Q=Q(P)$ is the matrix defined by \eqref{Q_def},
 and  $q^*,\hat{q}(\cdot)$ are the same as in \eqref{def:q-star}.


\subsection{Covering by independent sets}
Recall that, for  $\xvec = (x_1,\ldots,x_k)^T \in \pReals^k$,  we defined
\[ \lfloor \xvec\rfloor := (\lfloor x_1\rfloor, \ldots, \lfloor x_k\rfloor)^T
\qquad \text{and} \qquad  x_* := \min_{i \in [k]} x_i.
\]
Provided $x_*>0$,  we have,  for any $s>0$, 
\begin{equation}\label{eq: xvec}
s \|  \xvec \| \geq  \| \lfloor s \xvec\rfloor  \|  \ \geq \left( s-\frac{1}{x_*}\right)  \|  \xvec \|.
\end{equation}
Similarly, using {the monotonicity and scaling properties in Theorem} \ref{l:Qmatrix}(a), we get that
\begin{equation}\label{eq: xvec2}
s w(\xvec) \geq w( \lfloor s\xvec \rfloor)  \geq \left( s-\frac{1}{x_*}\right)  w(\xvec).
\end{equation}

{The next lemma shows that we can cover almost all  vertices of a random graph from the stochastic block model with the large "balanced" independent sets  provided by   Lemma~\ref{L:janson}.}

\begin{lemma}\label{L:upper}
	Let $\G \sim \SBM$,  where $P = (p_{ij})_{i,j\in [k]}$ is such that $p_{ij}=p_{ji}$ and $0\leq p_{ij}<1$ for all $i,j\in [k]$.  
	Let $\sigma \in [0,\sigma_0]$  for some fixed $0<\sigma_0<\frac14$.  
	Assume that     $w(\n)$
	and  $n_*$ satisfy the following as $\|\n\| \rightarrow \infty$: 
	\begin{align*}
	w (\n)  \geq  \|\n\|^{1-\sigma}, \qquad n_* = \|\n\|^{1+o(1)}, \qquad  n_* \gg \frac{w(\n)}{\log \|\n\|}. 
	\end{align*}
	Then, for any fixed constant $\eps\in {(0,1)}$, 
		with {probability
	at least 
	\[
	1 - \exp\left(-\|\n\|^{2-4\sigma+o(1)}\right),
	\] } 
	there is a colouring of  $\G$  with at most   
	$ \left(1 + \eps\right) \dfrac{w(\n)}{2\log (w(\n))}$ colours 
	covering at least 
	\[ 
	n_i \left(1 -  \frac{1}{ \log^2 \|\n\|} - 5 \eps^{-1} \frac{w(\n)}  {   n_*\log (w(\n))}\right)
	\]
	vertices from each block $B_i$ for all $i \in [k]$.
\end{lemma}

Throughout this section we let $\eps\in (0,1)$ be fixed and set 
\[
\nu: = (2-\eps )\frac{\log (w(\n))}{w(\n)}  \qquad  \text{and} \qquad 
\theta:=  \frac{ 1}{ 2\log^2 \|\n\|}  + 3\eps^{-1} \frac{w(\n)}  {   n_*\log (w(\n))}.
\]
To prove Lemma \ref{L:upper} we first claim some auxiliary results, whose proofs we defer to the end of this section. 

\begin{claim}\label{claim:U} 
	With probability at least
	\[
	1 - \exp\left(-\|\n\|^{2-4\sigma +o(1)}\right)
	\] 
	there exists a  sequence  $(U_1,\ldots, U_\ell)$ of  disjoint   independent sets  in $\G$ satisfying  the following.
	\begin{itemize}
		\item[(i)]  
		For all $j \in [\ell]$, we have 
		\[
		\b(U_j) = \left\lfloor \frac{\nu \|\n\|}{ \|\n^{(j)}\|} \n^{(j)} \right\rfloor,
		\]
		where 
		\begin{equation}\label{def_nj}
		\n^{(j)} = (n_1^{(j)},\ldots,n_k^{(j)})^T: = \n - \sum_{i = 1}^{j-1} \b(U_i) 
		\end{equation}
		satisfies   $\|\n^{(j)}\|\geq   \theta \|\n\| $.
		\item[(ii)] 
		The set $\bigcup_{j=1}^{\ell} U_j$ covers  all but  at most $\theta e^{\frac12} n_i$ vertices  from each  block $B_i$. 
		That is, for all $i\in [k]$, we have
		$\sum_{j=1}^{\ell} b_i(U_j) \ge    \left(1 - \theta e^{\frac12}\right) n_i.$
	\end{itemize}
\end{claim}

Our next claim gives the upper bound on the length of the sequence  $(U_1,\ldots, U_\ell)$ of disjoint   independent sets  in $\G$ from  Claim \ref{claim:U}.

\begin{claim} \label{claim:ell-bound}
	Suppose there exists a sequence $(U_1,\ldots,U_\ell)$ of disjoint independent subsets in $\G$  such that condition (i)  of Claim \ref{claim:U} holds. Then 
	$$\ell \ \leq\  \left(1+ \eps\right) \frac{w(\n)}{2\log (w(\n))}.$$
\end{claim}

We are ready to establish Lemma \ref{L:upper} based on the claims given above. 

\begin{proof}[Proof of Lemma \ref{L:upper}]
	We take the independent sets $U_1,\ldots, U_\ell$ provided by Claim \ref{claim:U} as our colour classes.  Note that  $\dfrac12 e^{\frac{1}{2}} \leq 1$  and $3\, e^{ \frac{1}{2}}  \leq 5$ so the condition~(ii)
	of  Claim \ref{claim:U}
	ensures that  this  colouring  
	covers  all but at most
	$
	n_i \left( \frac{1}{ \log^2 \|\n\|} + 5\eps^{-1} \frac{w(\n)}  {   n_*\log (w(\n))}\right)
	$
	vertices from each block $B_i$.    Claim \ref{claim:ell-bound} establishes the upper bound for the number of colours as desired.
\end{proof}

In the rest of this section we will prove first Claim \ref{claim:ell-bound} 
and then  Claim \ref{claim:U}.  To this end we  need the following lower  bounds on $	n_*^{(j)}$ {defined by 
$$   	n_*^{(j)}:= \min_{i\in [k]} n_i^{(j)}.$$}
\begin{claim} \label{claim:n*}
	Suppose there exists a sequence $(U_1,\ldots,U_\ell)$ of disjoint independent subsets in $\G$ such that condition (i)  of Claim \ref{claim:U} holds.  
	Then,  for all $j \in [\ell]$ such that $j \le  \dfrac{w(\n)}{\log (w(\n))}$, we have
	\[
	n_*^{(j)}  	\geq  \left(1-\dfrac\eps 3\right) n_* \frac{\|\n^{(j)}\|}{\|\n\|}. 
	\]
\end{claim}

\begin{proof}[Proof of Claim  \ref{claim:n*}] 
	It is sufficient to prove that, for all $j\in [\ell]$ such that   $j \le  \dfrac{w(\n)}{\log (w(\n))}$,
	\begin{equation}\label{eq:nja}
	n_*^{(j)} \geq  n_* \frac{\|\n^{(j)}\|}{\|\n\|}  - j +1.
	\end{equation}
	Indeed, by  the condition {(i)}  of Claim \ref{claim:U}  we have $  \|\n^{(j)}\|\geq   \theta \|\n\|$. 
	From the definition of $\theta $,  we get that, for every $j\in [\ell]$,
	\begin{equation} \label{eq:nj}
	\|\n^{(j)}\|\geq  \theta \|\n\|  =        {\frac{\|\n\|}{ 2\log^2 \|\n\|}}+3\eps^{-1} \frac{ \|\n\| w(\n)}{n_*\log(w(\n))}.
	\end{equation}
	{Using  the trivial bound $ \|\n\| \ge n_*$, we immediately get   from \eqref{eq:nj}  that}
	\begin{equation}\label{eq:nj-lower}
	n_* \frac{\|\n^{(j)}\|}{\|\n\|} \geq  
	3 \eps^{-1} \frac{  w(\n)}{ \log(w(\n))}.
	\end{equation}
	Thus, if $j \leq     \dfrac{w(\n)}{\log (w(\n))}$ then Claim \ref{claim:n*} follows from  \eqref{eq:nja} {and \eqref{eq:nj-lower}}.
	
	We will prove  \eqref{eq:nja} by induction on $j$. 
	Clearly,  it is true for $j=1$ since $\n^{(1)} = \n$.   Suppose, we established the  claim for $j=i$  such that $i<\ell$.    By definition and using \eqref{eq: xvec} with 
	$\x =\n^{(i)}$ and   $s=\frac{\nu \|\n\|}{ \|\n^{(i)}\|}$,  we get that 
	\begin{equation}\label{eq:ni-ratio-upper}
	\frac{\|\n^{(i+1)}\|}{\|\n^{(i)}\|} \leq 1 - \frac{\nu \|\n\|}{ \|\n^{(i)}\|} + \frac{1}{n_*^{(i)}}. 
	\end{equation}
	{Combining the induction hypothesis, the bound of \eqref{eq:ni-ratio-upper}, and   $\frac{\|\n^{(i+1)}\|}{\|\n^{(i)}\|} \leq 1$,}
  we find that
	\begin{align*}
	n_*^{(i+1)} &\geq  \left(1 - \frac{\nu \|\n\|}{ \|\n^{(i)}\|} \right) n_*^{(i)}
	\geq  \left(\frac{\|\n^{(i+1)}\|}{\|\n^{(i)}\|} - \frac{1}{n_*^{(i)}} \right) n_*^{(i)}
	\\ 
	&\geq  n_* \frac{\|\n^{(i+1)}\|}{\|\n\|}  - (i -1) \frac{\|\n^{(i+1)}\|}{\|\n^{(i)}\|}  -1
	\geq n_* \frac{\|\n^{(i+1)}\|}{\|\n\|} - i.
	\end{align*}
	Note also that the induction hypothesis and  \eqref{eq:nj-lower} imply that
	$n_*^{(i)}$ is positive, since it is  at least $ \left(1-\dfrac\eps 3\right) n_* \frac{\|\n^{(i)}\|}{\|\n\|}$. 
	Thus, the claim is true for $j=i+1$ and, by induction, for all $j \in [\ell]$ such that
	$j \le  \dfrac{w(\n)}{\log (w(\n))}$.  
\end{proof}

\begin{proof}[Proof of Claim  \ref{claim:ell-bound}] 
	Assume otherwise that $\ell > \left(1 + \eps\right) \dfrac{w(\n)}{2\log (w(\n))}$. 
	By definition  \eqref{def:b(U)} we have $|U_j| = \|\b(U_j)\| $ and by Claim \ref{claim:U}{(i)}   and \eqref{eq: xvec} with $\xvec =\n^{(j)}$ and $s=\frac{\nu \|\n\|}{ \|\n^{(j)}\|}$ we have 
	\begin{equation*} 
	\|\b(U_j)\| \geq 
	\left( \frac{\nu \|\n\|}{ \|\n^{(j)}\|}  - \frac{1}{ n_*^{(j)}}\right)  \|\n^{(j)}\|.
	\end{equation*}
	Note further that, by definition of $\nu$ and the assumptions $ n_* \gg \frac{w(\n)}{\log \|\n\|}$, 
	$w(\n) \geq \|\n\|^{1-\sigma}$, we have $\nu \gg 1/n_*$. 
	Then, using Claim  \ref{claim:n*}, we get  that 
	$\frac{\nu \|\n\|}{ \|\n^{(j)}\|}   \gg  \frac{1}{ n_*^{(j)}}$
	for all $j\in [\ell]$    such that
	$j \le  \dfrac{w(\n)}{\log (w(\n))}$. Therefore,
	\begin{equation} \label{eq:|U_j|}
	|U_j| = \|\b(U_j)\| \geq 
	\left( \frac{\nu \|\n\|}{ \|\n^{(j)}\|}  - \frac{1}{ n_*^{(j)}}\right)  \|\n^{(j)}\| 
	=  (1-o(1)) \nu \|\n\|.
	\end{equation}  
	Using \eqref{eq:|U_j|} and our assumption that  $\ell > \left(1 + \eps\right) \dfrac{w(\n)}{2\log (w(\n))}$, we get
	\begin{align*}
	\|\n\|\geq \sum_{j \in [\ell]} |U_j| 
	& \geq    (1-o(1))  \nu \|\n\|   \left(1 +\eps\right) \frac{ w(\n)}{ 2\log (w(\n))}
	\\ &=(1-o(1))(2 - \eps) \left(1 +\eps\right) \frac{\|\n\|}{2}> \|\n\|.
	\end{align*}
{The last inequality is true for any fixed $\eps \in (0,1)$ when $o(1)$ gets sufficiently small.}  
	This contradiction proves   Claim  \ref{claim:ell-bound}.
\end{proof}

\begin{proof}[Proof of Claim  \ref{claim:U}]
	In order to show the existence  of  such a sequence  
	$(U_1,\ldots, U_\ell)$ of  disjoint    independent sets in $\G$, we repeatedly apply Lemma~\ref{L:janson}.        
	Suppose we already constructed sets $U_1, \ldots, U_{j-1}$. 
	We will show that if  $\|\n^{(j )}\| \geq \theta \|\n\|$ then, with sufficiently high probability, we can  find  another independent set $U_{j }$ in the induced subgraph of $\G$ on remaining  vertices,  
	which  satisfies condition (i) of  Claim  \ref{claim:U}. 
	By Claim \ref{claim:ell-bound}, we get that 
	\begin{equation}\label{jj_bound}
	j  \leq  \left(1 +\eps \right) \frac{w(\n)}{ 2\log (w(\n))} \leq  \frac{w(\n)}{ \log (w(\n))}. 
	\end{equation}
	For all $i< j$,
	using \eqref{eq: xvec} with 
	$\x =\n^{(i)}$ and   ${s}=\dfrac{\nu \|\n\|}{ \|\n^{(i)}\|}$,   we find that
	\begin{equation}\label{eq:ni-ratio}
	\frac{\|\n^{(i+1)}\|}{\|\n^{(i)}\|} \geq 1 - \frac{\nu \|\n\|}{ \|\n^{(i)}\|}. 
	\end{equation}
	From \eqref{eq: xvec2} (with the same $\x$ and $s$), we obtain
	\begin{equation}\label{eq:wj}
	w(\n^{(i+1)}) \leq \left(1 - \frac{\nu \|\n\|}{ \|\n^{(i)}\|} + \frac{1}{n_*^{(i)}} \right)  w(\n^{(i)}).
	\end{equation}
{Due to \eqref{jj_bound}, we can apply  Claim \ref{claim:n*} to obtain   $$n_*^{(i)}  	\geq  \left(1-\dfrac\eps 3\right) n_* \frac{\|\n^{(i)}\|}{\|\n\|} \geq  \frac{3-\eps}{\eps} \frac {w(\n)}{ \log(w(\n))}\frac{\|\n^{(i)}\|}{\|\n^{(i+1)}\|},$$ 
		where the last inequality follows from  \eqref{eq:nj-lower}  by taking $j=i+1$ which gives 
		\[ \frac{n_*}{\|\n\|} \geq  \frac{3}{\eps} \frac{w(\n)}{ \log(w(\n))} \frac{1}{\|\n^{(i+1)}\|}.\]}
	Then, combining this with \eqref{eq:ni-ratio}, we have
	\[
	1 - \frac{\nu \|\n\|}{ \|\n^{(i)}\|} + \frac{1}{n_*^{(i)}}   
	\leq     \frac{\|\n^{(i+1)}\|}{\|\n^{(i)}\|} \left( 1+ \frac{\eps \log(w(\n))}{(3-\eps)w(\n)}\right),
	\]
	which in  \eqref{eq:wj} implies   
	\begin{equation}\label{eq:upperw}
	\frac{w(\n^{(i+1)})}{w(\n^{(i)})} \leq  \frac{\|\n^{(i+1)}\|}{\|\n^{(i)}\|} \left( 1+ \frac{\eps \log(w(\n))}{(3-\eps)w(\n)}\right).
	\end{equation}
	Multiplying \eqref{eq:upperw} together for $i =1,\ldots, j-1$, we obtain
	\begin{align}\label{eq:upperwj}
	\frac{w(\n^{(j )})}{w(\n)} &= {\prod_{i=1}^{j-1} 	\frac{w(\n^{(i+1)})}{w(\n^{(i)})}}
	\leq  \frac{\|\n^{(j )}\|}{\|\n\|} \left( 1+ \frac{\eps \log(w(\n))}{(3-\eps)w(\n)}\right)^{j-1 }\nonumber\\  
	&\leq \frac{ \|\n^{(j )}\|}{\|\n\|} 
	\exp\left( \frac{\eps( 1 + \eps)}{2(3-\eps)}\right)
	\leq \frac{ \|\n^{(j )}\|}{\|\n\|} e^{\frac{\eps}{2}}, 
	\end{align}
	{where the penultimate inequality is due to $1+x\le e^x$ and the first inequality in \eqref{jj_bound}.} 
	Similarly to \eqref{eq:upperw} and \eqref{eq:upperwj}, we get
	\[
	\frac{w(\n^{(i+1)})}{w(\n^{(i)}) } \geq \left( \frac{\|\n^{(i+1)}\|}{\|\n^{(i)}\|}  - \frac{1}{n_*^{(i)}} \right)  \geq    \frac{\|\n^{(i+1)}\|}{\|\n^{(i)}\|} \left( 1 - \frac{\eps \log(w(\n))}{(3-\eps)w(\n)}\right), 
	\]
	which leads to the bound
	\begin{equation}\label{eq:lowerw}
	\frac{w(\n^{(j )})}{w(\n) } \geq    \frac{\|\n^{(j )}\|}{\|\n\|}e^{- \frac{\eps}{2}}. 
	\end{equation}
	Thus, we obtain from \eqref{eq:upperwj} and \eqref{eq:lowerw} that
	\begin{equation}\label{eq:upperlowerw}
	\frac{\|\n^{(j )}\|}{\|\n\|}e^{- \frac{\eps}{2}} \leq \frac{w(\n^{(j )})}{w(\n) } \leq   \frac{ \|\n^{(j )}\|}{\|\n\|} e^{\frac{\eps}{2}}. 
	\end{equation}
	From  \eqref{eq:nj},   we have $\dfrac{\|\n^{(j )}\|}{\|\n\|}\geq \theta \geq \dfrac{1}{2\log^2\|\n\|}$ and obviously $ \dfrac{\|\n^{(j )}\|}{\|\n\|}\leq 1$.
	Using the assumption  $n_* = \|\n\|^{1+o(1)}$,	
	we get that
	\begin{equation}\label{eq:asympw}
	\log(\|\n\|) =	 (1+o(1))\log \|\n^{(j)}\| .
	\end{equation}
	By Claim \ref{claim:n*} and the assumption $ n_* \gg \frac{w(\n)}{\log \|\n\|}$  	 
	we find that  
	$$n_*^{(j)}  	\geq  \left(1-\dfrac\eps 3\right) n_* \frac{\|\n^{(j)}\|}{\|\n\|} 	\gg  \frac{\|\n^{(j )}\|}{\|\n\|}\cdot  \frac{ w(\n)}{\log \|\n\|)} 
	\geq e^{-\frac{\eps}{2}}   \frac{ {w(\n^{(j )})}}{\log \|\n\|},
	$$
	where the last inequality follows from \eqref{eq:upperwj}.         
	Combining this with \eqref{eq:asympw}, we get
	\[
	n_*^{(j)} 
	\gg  \frac{ w(\n^{(j)})}{ \log \|\n^{(j )}\|}. 
	\]
	Combining, Claim \ref{claim:n*},  the assumption $n_* = \|\n\|^{1+o(1)}$,  
	and \eqref{eq:asympw}, we find  that
	\[
	1\geq 	\frac{ 	n_*^{(j)} }
	{\| 	
		\n^{(j)} \|} 
	\geq \left(1-\dfrac\eps 3\right)  \frac{n_*}{\|\n\|} =  \|\n\|^{o(1)} = \|\n^{(j)}\|^{o(1)}.
	\]
	Furthermore,  the assumption $w(\n) \geq \|\n\|^{1-\sigma}$   together with {\eqref{eq:upperlowerw} and}   \eqref{eq:asympw} implies
	\[
	w(\n^{(j)} ) \geq  w(\n) \frac{\|\n^{(j)}\|}{\|\n\|}e^{- \frac{\eps}{2}} \geq   e^{- \frac{\eps}{2}} \|\n\|^{1-\sigma - o(1)} =    \|\n^{(j)}\|^{1-\sigma - o(1)}.
	\]
	Thus, all assumptions of Lemma \ref{L:janson} hold for the random graph $\G^{(j)} \sim \mathcal G(\n^{(j)},P)$.

	Applying Lemma \ref{L:janson} to $\G^{(j)} \sim \mathcal G(\n^{(j)},P)$ we show the existence of an   independent set 
	$U'\subset V(\G^{(j)})$  in $\G^{(j)}$ with probability 
	at least 
	$1 - \exp\left(-\|\n\|^{2-4\sigma +o(1)}\right)$
	such that 
	$
	\b(U') = \lfloor\nu' \n\rfloor 
	$,
	where 
	\[
	\nu' = (2+o(1)) \frac{\log \left(w(\n^{(j)}) \right)}{w(\n^{(j)})}.
	\]
	Moreover,  since  $2-4\sigma >1$, the probability  that there exists 
	$W \subseteq V(\G)$ such that $\b(W) = \n^{(j)}$ and $\G[W]$ does not contain  such an
	independent set $U'$  is at most
	\[
	\binom{\|\n\|}{\|\n^{(j)}\|}  \exp\left(-\|\n\|^{2-4\sigma +o(1)}\right)
	= \exp\left(-\|\n\|^{2-4\sigma +o(1)}\right).
	\] 
	In particular, we get that the graph obtained from $\G$ by removing $U_1,\ldots, U_{j-1}$ 
	contains such $U'$ with probability at least  $ 1 - \exp\left(-\|\n\|^{2-4\sigma +o(1)}\right).$

	Next, we show that it is  possible to  find   $U_j \subseteq U'$  such that $\b(U_j) = \left\lfloor \dfrac{\nu \|\n\|}{ \|\n^{(j)}\|} \n^{(j)} \right\rfloor$.  To do this, it is sufficient to show that 
	$\nu' \geq  \dfrac{\nu \|\n\|}{ \|\n^{(j)}\|}$.  
	Using \eqref{eq:upperlowerw}, \eqref{eq:asympw}, and the assumption $w(\n) \geq \|\n\|^{1-\sigma}$, we find that
	\[
	\log \left(w(\n^{(j)}) \right)= (1+o(1))  \log \left(w(\n) \right).
	\]
	Observe that $g(\eps):= (2-\eps)e^{\eps/2}$ is decreasing on $\pReals$, so $g(\eps)<g(0)=2$.  Therefore,  using  \eqref{eq:upperlowerw}  and the first inequality in  \eqref{eq: xvec2}, we get that 
	\[
	\nu' \geq  (2-\eps)e^{\eps/2} \frac{\log \left(w(\n) \right)}{w(\n^{(j)})}
	\geq  (2-\eps )\frac{ \|\n\|\log (w(\n))}{ \|\n^{(j)}\| w(\n)} = \frac{\nu \|\n\|}{ \|\n^{(j)}\|}.
	\]

	The probability that there exists the required  sequence $(U_1,\ldots, U_{\ell})$
	can be estimated as follows.  Using  Claim \ref{claim:ell-bound}   and applying the union bound
	for the event that there is no suitable choice for 
{$U_{j+1}$} after removing $U_1,\ldots, U_j$ from $\G$,  we get that  
	\begin{align*}
	\sum_{j \in [\ell]}  \exp\left(-\|\n\|^{2-4\sigma +o(1)}\right) 
	&\leq    \left(\frac12 + \eps\right) \dfrac{w(\n)}{\log (w(\n))}  \exp\left(-\|\n\|^{2-4\sigma +o(1)}\right) 
	\\ &=  \exp\left(-\|\n\|^{2-4\sigma +o(1)}\right). 
	\end{align*}
	To derive the last inequality, we  use the assumptions to estimate 
	$\dfrac{w(\n)}{\log (w(\n))} \ll n_* \leq \|\n\| $  and recall that $2-4\sigma>1$.
	
	The construction of  the sequence $(U_1,\ldots, U_{\ell})$ is terminated when 
	$\|\n^{(\ell+1)}\|< \theta \|\n\|$. 
	Note that, for any $i\in[k]$,
	\[
	n_i^{(j+1)} \leq n_i^{(j)}- \frac{\nu \|\n\|}{\|\n^{(j)}\|} n_i^{(j)} +1 
	\leq \left(1 - \frac{\nu \|\n\|}{\|\n^{(j)}\|}  + \frac{1}{n_*^{(j)}} \right)  n_i^{(j)}.
	\]
	Repeating the arguments of \eqref{eq:upperw} and \eqref{eq:upperwj}, we find that 
	\[
	n_i^{(\ell+1)} \leq e^{\frac{\eps}{2}} \frac{\|\n^{(\ell+1)}\|}{\|\n\|}  n_i
	\leq   \theta  e^{\frac{1}{2}}n_i.
	\]
	Thus,  condition (ii) of  Claim  \ref{claim:U} is satisfied. This completes the proof of  Claim  \ref{claim:U}.
\end{proof}

\subsection{{Final ingredient for} colouring completion}
{In this section, we combine Lemma \ref{L:upper} and Theorem \ref{T:uppercrude}  to estimate the chromatic number of $\G \sim \SBM$ under the additional condition that $w_*(\n)$ is asymptotically equal to $w(\n)$.  
In the general case of Theorem \ref{T:upper}, this additional condition will be satisfied by each part of the random graph $\G$ corresponding to a near-optimal integer system given by Theorem \ref{l:Qmatrix}(g); see Section \ref{S:upper}.}

\begin{lemma}\label{L:upper2}
		Let $\G \sim \SBM$,  where $P = (p_{ij})_{i,j\in [k]}$ is such that $p_{ij}=p_{ji}$ and $0\leq p_{ij}<1$ for all $i,j \in [k]$. Let $\sigma \in [0,\sigma_0]$  for some fixed $0<\sigma_0<\frac14$.  
	Assume   that, as $\|\n\| \rightarrow \infty$:
	\[
	k   = \|\n\|^{o(1)}, \qquad     \log\|\n\|  \gg k q^*   \geq  {\|\n\|^{-\sigma}.}
	\]
	Assume also that 
	\[
	w_*(\n)  = (1+o(1)) w(\n)  \geq ( q^*\|\n\|)^{1+o(1)}.
	\]
	Then,    for any fixed $\eps\in (0,1)$,
	\[
	\Pr\left( \chi(\G) >  \left(1 +\eps \right) \frac{w(\n)}{ 2\log (q^*\|\n\|)}
	+     20 \eps^{-2}   \frac{k \hat q(\n) q^* \|\n\|  }{\log^2 (q^*\|\n\|)}
	\right)  \leq \exp \left(- \|\n\|^{2-4\sigma +o(1)}\right).
	\]
\end{lemma}    
\begin{proof}
	Let 
	\[
	n_0:= \dfrac{\eps\, w(\n)}{2 k q^*}. 
	\]
	Consider   the  vector $\tilde{\n} = (\tilde n_1,\ldots, {\tilde n_k})^T \in \Naturals^k$ defined by
	\[
	\tilde n_i:= 
	\begin{cases}
	n_i, &\text{if } n_i \geq  n_0,\\
	0, &\text{otherwise.}
	\end{cases}
	\]
	Let $U_{\text{big}}$ be the union of blocks $B_i$ for which $n_i \geq n_0$.
	We will apply  Lemma \ref{L:upper} for the induced subgraph 
	$\widetilde{\G}:= \G[U_{\text{big}}]  \sim \mathcal G(\tilde \n, P)$ (ignoring zero components of $\tilde \n$).
	Then, we will use Theorem~\ref{T:uppercrude} to colour the rest of the vertices of $\G$.

	First, we check that  $\widetilde{\G}$ satisfies the  assumptions of   Lemma \ref{L:upper}.
	From {the triangle inequality in Theorem \ref{l:Qmatrix}(e)}  and the assumptions, we find that  
	\[
	w(\tilde \n) \geq 
	w_*(\tilde \n)  \geq w_*({\n}) -  w_*(\n-\tilde{\n})
	\geq  (1+o(1)) w(\n) - w_*(\n-\tilde{\n}).
	\]
	Using {the upper bound of Theorem} \ref{l:Qmatrix}(d) and by the definitions of $\tilde{\n}, n_0$, we get  
	\[
	w_*(\n-\tilde{\n}) \leq  \sum_{i\in [k]}  n_0 q_{ii}  
	{=\sum_{i\in [k]}  \dfrac{\eps\, w(\n)}{2 k q^*} q_{ii}    } \leq  \dfrac\eps  2  w(\n).
	\]
	Therefore, by the assumptions
	\begin{equation}\label{eq:wtilde}
	w(\tilde \n)  \geq \left(1 -   \frac\eps 2 +o(1)\right) w(\n) \geq (q^* \|\n\|)^{1+o(1)}\geq \|\n\|^{ 1-\sigma +o(1)}
	\geq  \|\tilde\n\|^{ 1-\sigma +o(1)}.
	\end{equation}
	{Using our assumption that $w(\n) = (1+o(1)) w_*(\n)$ and   Theorem \ref{l:Qmatrix}(d) again, we get 
	\begin{equation}\label{eq:n-ntilde}
	\|\n - \tilde \n\| \leq  
	kn_0  =  \dfrac{\eps\, w(\n)}{2 q^*} 
 =   (1+o(1))\dfrac{\eps\, w_*(\n)}{2 q^*}
	\leq (1+o(1))\dfrac{\eps}{2} \|\n\|.
	\end{equation}}
	In particular, we get $\|\tilde \n\| = \|\n\|^{1+o(1)}$. 
	By the definition of $\tilde{\n}$, all non-zero components of $\tilde{\n}$ are at least $n_0$.  
	Using  {the bounds of  Theorem} \ref{l:Qmatrix}(d)  and the assumptions, we have that 
	\[q^*\|\n\|\geq  
 	w_*(\n)  \geq  (1+o(1))w(\n) \geq  (q^* \|\n\|)^{1+o(1)}.
	\] 
	Thus,  $\frac{w(\n)}{q^*\|\n\|} = \|\n\|^{o(1)}$.
	Recalling also  $k=\|\n\|^{o(1)}$, we find that 
	\begin{equation}\label{eq:n0}
	n_0 ={\dfrac{\eps\, w(\n)}{2 k q^*} = } \dfrac{\eps}{2k}\cdot \dfrac{ w(\n)}{q^*\|\n\|} \cdot  \|\n\| = \|\n\|^{1+o(1)} =  \|\tilde \n\|^{1+o(1)} 
	\end{equation}
	and, since  $kq^* \ll \log \|\n\|$ and $\tilde{\n} \preceq \n$, 
	\[
	n_0 \gg  \dfrac{ w(\n)}{\log \|\n\|} \geq 
	 \dfrac{ w(\tilde \n)}{\log \|\n\|}   = \dfrac{w(\tilde\n)}{ (1+o(1))\log \|\tilde{\n}\|}.
	\]
	Thus, all  assumptions of   Lemma \ref{L:upper} for $\widetilde{\G} \sim \mathcal G(\tilde{\n}, P)$ hold. 
	
	Applying Lemma \ref{L:upper} with $\tilde\eps := \dfrac\eps 3$,
{	we show that,
	with probability
	at least 
	\[
	1 - \exp\left(-\|\tilde  \n\|^{ 2-4\sigma +o(1)}\right) =
	1 - \exp\left(-\| \n\|^{  2-4\sigma +o(1)}\right),
	\]  
	there is a colouring of  $\widetilde{\G} $  with at most   
	\[
	\left(1 + \tilde\eps\right) \dfrac{w(\tilde \n)}{2\log (w(\tilde \n))}
	\leq  \left(1+ \tilde\eps +o(1) \right) \dfrac{w(\n)}{2\log (q^*\|\n\|)}
	\]
	colours 
	covering  all vertices from each block $B_i$ that $n_i \geq n_0$ except at most 
	\[ 
	n_i \left( \dfrac{1}{ \log^2 \|\tilde \n\|} + 5(\tilde\eps)^{-1} \dfrac{w(\tilde \n)}  {   n_0\log (w(\tilde \n))}\right)
	\] 
	vertices.}
	
	Using \eqref{eq:wtilde} and  recalling $\sigma<\dfrac14$, we find that 
	$w(\tilde \n)  = (w(\n))^{1+o(1)} = (q^*\|\n\|)^{1+o(1)}$.
	Using also  the assumption  $k q^* \ll  \log\|\n\| $, we conclude that the set of remaining uncoloured vertices (with sufficiently high probability) has at most 
	\begin{align*}
	u_i &:= n_0 + n_i \left( \frac{1}{ \log^2 \|\tilde \n\|} + 5(\tilde\eps)^{-1} \frac{w(\tilde \n)}  {   n_0\log (w(\tilde \n))}\right)
	\\
	& = n_0 +  (1+o(1)) \frac{ 30 kq^* n_i  }  {\eps^2  \log (q^*\|\n\|)} 
	\end{align*}
	vertices in each block $B_i$.  Since $n_0 = \|\n\|^{1+o(1)} $   by \eqref{eq:n0} and $k = \|\n\|^{o(1)}$ by our assumptions, we find that $u_i = \|\n\|^{1+o(1)}$
	{and $\|\u\| = \|\n\|^{1+o(1)}$, where $\uvec = (u_1,\ldots,u_k)^T$.  Then,   we get that
	\[
      q^*  \geq \hat q(\uvec)  :=\frac{\sum_{i\in [k]} u_i q_{ii}}{\|\uvec\|} \geq \frac{q^* n_0}{\|\uvec\|}  = q^* \|\n\|^{o(1)}.
	\]}
	{By the definitions of $u_i$ and $n_0$}, we observe  that 
	\begin{align*}
	\hat q(\uvec)\|\uvec\|   &\leq  kq^* n_0  + (1+o(1)) \frac{ 30 kq^*  \|\n\| \hat q(\n)  }  {\eps^2    \log (q^*\|\n\|)} 
	\\
	&= \frac{\eps}{2} w(\n)+  (1+o(1))   \frac{ 30 kq^* \|\n\| \hat q(\n)  }  {\eps^2    \log (q^*\|\n\|)}. 
	\end{align*}
	Using Theorem \ref{T:uppercrude} with $n:=\|\n\|$ with any $\eps'< \dfrac13$, we can colour the remaining vertices 
	using at most 
	\[
	\left(1+ \eps' \right) \frac{\hat q(\uvec)\|\uvec\| }{2\log (q^*\|\n\|)} 
	\leq    \frac{(1+\eps')\eps}{4} \cdot \frac{  w(\n)  }{ \log  (q^* \|\n\|)}
	+   20 \eps^{-2}  \frac{kq^*  \|\n\| \hat q(\n) }  { \log  (q^* \|\n\|)}
	\]
	colours with probability at least 
	$
	1 - \exp\left(-\| \n\|^{2-4\sigma +o(1)}\right).
	$  
	Thus, the total number of colours is  at most 
	\begin{align*}
	\left(1+  \tilde{\eps} + \frac{(1+\eps')\eps}{2}+o(1)  \right)  \frac{w(\n)}{2\log (q^*\|\n\|)}
	+   20 \eps^{-2}  \frac{kq^*  \|\n\| \hat q(\n) }  { \log  (q^* \|\n\|)}.
	\end{align*} 
	The claimed bound on $\chi(\G)$ follows since $\tilde{\eps} = \eps/3$ and $\frac{(1+\eps')\eps}{2}<2\eps/3$.
\end{proof}


\subsection{Upper tail bound: proof of Theorem \ref{T:upper}}\label{S:upper}

{By the near-optimal integer system property given in  Theorem \ref{l:Qmatrix}(g), we can find $k$ vectors  $(\n^{(t)})_{t\in [k]}$ from $\Naturals^k$
such that 
\begin{equation}\label{nt-def}
\n =  \sum_{t \in [k]}\n^{(t)} \qquad \text{and} \qquad 	 \sum_{t \in [k]} w(\n^{(t)}) \leq w_*(\n)+ k^2 q^*. 
\end{equation}
We treat our graph $\G$ as the union of the vertex disjoint random graphs $\G^{(t)}\sim \mathcal G(\n^{(t)}, P)$, for $t\in [k]$.
Since we can colour them with different colours, we have that, with probability $1$,
\begin{equation}\label{chi-inequality}
\chi(\G) \leq  \sum_{t\in [k]} \chi(\G^{(t)}).
\end{equation}
Let  
\[
T_{\text{small}} =\left\{t\in [k] :  w(\n^{(t)}) < \frac{w_*(\n)}{k^2 \log \|\n\|} \right\}.
\] 
The proof of Theorem \ref{T:upper} consists of two parts. First, applying Theorem    \ref{T:uppercrude}, we show that, with sufficiently high probability,    
$\sum_{t\in T_{\rm small}} \chi(\G^{(t)})  \ll \frac{w_*(\n)}{\log (q^*\|\n\|)} $. Second, we use  Lemma \ref{L:upper2} to estimate  $ \chi(\G^{(t)})$ for $t \notin T_{\rm small}$.}

{Before proceeding, we derive some preliminary bounds implied by our assumptions. Since 
	$k   = \|\n\|^{o(1)}$ and
	$\hat q(\n), q^* =  \|\n\|^{-\sigma+o(1)}$,  
	we find that
	\[
	\frac{ (\hat q(\n) )^2  }{k   q^* } 
	=
	\|\n\|^{-\sigma+o(1)}
	=
	\|\n\|^{o(1)} kq^*.
	\]}
Then, using {the bounds of Theorem}~\ref{l:Qmatrix}(d), we get
\begin{equation}\label{eq:mubounds}
	q^*  \geq  \hat q(\n)  \geq \frac{w_*(\n)}{\|\n\|} \geq
	\frac{ (\hat q(\n) )^2 }{\sum_{i\in [k]} q_{ii}} \geq 
	\frac{ (\hat q(\n) )^2  }{k   q^* }
	=
	\|\n\|^{o(1)} kq^*.
\end{equation}
Since $kq^*\|\n\| = \|\n\|^{1-\sigma+o(1)}$   we derive from \eqref{eq:mubounds} that
\begin{equation}\label{eq:muq*}
	{w_*(\n) = (q^*\|\n\|)^{1+o(1)}  =\|\n\|^{1-\sigma  +o(1)}} \gg  k^2 \log (q^*\|\n\|).
\end{equation}
Using \eqref{eq:mubounds} and {assumption \eqref{ass_2}}, we  get that 
\[
\hat q(\n) \geq \frac{w_*(\n)}{ \|\n\|} \gg 
\frac{kq^* \hat q(\n) }{ \log \|\n\|},
\]
which implies 
\[
q^* \leq kq^* \ll \log \|\n\|,
\]
{which is needed to apply Theorem    \ref{T:uppercrude} and Lemma \ref{L:upper2}. }  
Also, by the definition of $w_*(\cdot)$, we know that
\[
w_*(\n)   \leq 
\sum_{t\in [k]}w_*(\n^{(t)}).
\
\]
which, together with  \eqref{nt-def},  implies  that 
\[
      \sum_{t\in [k]} (w (\n^{(t)})  - w_*(\n^{t})) \leq k^2 q^*
\]
Since the every term of the sum above is non-negative, 
using  the assumption $k   = \|\n\|^{o(1)}$ and the estimate  $q^*  \ll \log \|\n\|$, we derive that, for any $t\in [k]$, 
\[
	w(\n^{(t)}) \leq w_*(\n^{(t)}) + k^2q^* = w_*(\n^{(t)}) +  \|\n\|^{o(1)}.
\]
Then,  using {the first equality of \eqref{eq:muq*} and our assumption $k   = \|\n\|^{o(1)}$}, for any $t \in [k]\setminus T_{\rm small}$,  we  get
\begin{equation}\label{wnt}
	w(\n^{(t)}) \geq \frac{w_*(\n)}{k^2 \log \|\n\|}  = (q^*\|\n\|)^{1+o(1)} \geq (q^*\|\n^{(t)}\|)^{1+o(1)}.
\end{equation}
This implies that 
\[
		w_*(\n^{(t)}) = (1+o(1)) w(\n^{(t)}) \geq  (q^*\|\n^{(t)}\|)^{1+o(1)}
\]
as required by Lemma \ref{L:upper2}.
 
Now, consider any $t\in T_{\text{small}}$. Define   {$\uvec = (u_1,\ldots, u_k)^T \in \pReals^k$ }
\[
u_i:= n_i^{(t)} +  \frac{w_*(\n)}{k^2 q^* \log \|\n\|}. 
\] 
Using   {the bound  $w_*(\n) \leq q^* \|\n\|$ of Theorem~\ref{l:Qmatrix}(d) and the  assumption $k=\|\n\|^{o(1)}$, we get 
\begin{equation}\label{eq:u-n}
	\begin{aligned}
\|\n^{(t)}\| \leq \|\uvec\|  
&\leq \|\n^{(t)}\|+ \sum_{i\in [k]}\frac{w_*(\n)}{k^2 q^* \log \|\n\|}
\\ &\leq \|\n^{(t)}\|+ \frac{\|\n\|}{k  \log \|\n\|}
\leq  \|\n\|+\|\n\|^{1+o(1)}
=
\|\n\|^{1+o(1)}.
\end{aligned}
\end{equation}}
Using   {the lower bounds of Theorem~\ref{l:Qmatrix}(d) and the inequality $\hat q(\n^{(t)}) \leq q^*$}, we find that 
\[
w(\n^{(t)}) \geq  w_*(\n^{(t)}) \geq  \frac{(\hat q(\n^{(t)}))^2 \|\n^{(t)}\| }{kq^*} \geq \frac{\hat q(\n^{(t)})\|\n^{(t)}\|}{k}.
\]
{This implies $\hat{q}(\n^{(t)})\|\n^{(t)}\| \leq k w(\n^{(t)}) < \frac{w_*(\n)}{k \log \|\n\|} $ because  $w(\n^{(t)}) < \frac{w_*(\n)}{k^2 \log \|\n\|} $ since $t \in T_{\rm small}$. Using the definition of  $u_i$ we get  that}
\begin{equation}\label{hatq-u}
	\begin{aligned}
		\hat{q}(\uvec)\|\uvec\|  &= {\sum_{i\in [k]} q_{ii} u_i = 
			\sum_{i\in [k]} q_{ii}  \left(n_i^{(t)}   +\frac{w_*(\n)}{k^2 q^* \log \|\n\|}\right)} \\ &=
		\hat{q}(\n^{(t)})\|\n^{(t)}\| +  \frac{w_*(\n)}{k^2 q^* \log \|\n\|} \sum_{i\in [k]} q_{ii}  
		\ll   \frac{ w_*(\n)}{k  }.
	\end{aligned}
\end{equation}
{Using \eqref{eq:muq*}, the inequality $\hat{q}(\uvec) \leq q^*$, and  $q^* = \|\n\|^{-\sigma+o(1)}$
	by \eqref{ass_sbm}, observe also that
	\[
	\hat{q}(\uvec)\|\uvec\|   \geq \frac{w_*(\n)}{k^2  q^* \log \|\n\|} \sum_{i\in [k]} q_{ii} \geq  \frac{w_*(\n)}{k^2 \log \|\n\|}
	=q^*\|\n\|^{1+o(1)}.
	\]
	Recalling from \eqref{eq:u-n}  that $\|\uvec\| \leq \|\n\|^{1+o(1)}$ 
		and using $\hat{q} (\uvec) \leq q^*$, we derive that 
 \[ q^*  = \hat{q} (\uvec) \|\n\|^{o(1)}.
\]}

Applying Theorem  \ref{T:uppercrude}
with $n:=\|\n\|$ {and using $\hat{q}(\uvec)\|\uvec\|   \ll   \frac{ w_*(\n)}{k  }$ from \eqref{hatq-u}} we get that 
\[
\chi(\G^{(t)}) ={O \left( \frac{\hat{q}(\uvec) \|\uvec\|}{2 \log (q^*\|\n\|)}\right)} \ll    \frac{ w_*(\n)}{k  \log (q^* \|\n\|)}
\]
with probability at least $1 - {\exp\left(-\|\n\|^{ 2-4\sigma +o(1)}\right)}$.  Applying the union bound, it follows that, with sufficiently high probability,
\begin{equation}\label{chi-small}
	\sum_{t \in T_{\text{small}}} \chi(\G^{(t)})   \ll   \frac{w_*(\n)}{\log (q^*\|\n\|)}.
\end{equation}

Next, we  consider any    $t \in [k] \setminus  T_{\text{small}}$. 
{Since  $w_*(\n^{(t)}) \leq q^*\|\n^{(t)}\| $    by Theorem \ref{l:Qmatrix}(d) and $w_*(\n^{(t)}) = (1+o(1)) w(\n^{(t)}) $, we have  $\|\n^{(t)}\|  \geq (1+o(1))  \frac{w (\n^{(t)})}{q^*}$. Using   \eqref{wnt} and  the bound $k=\|\n\|^{o(1)}$,} we find also that 
{
	\[
	\|\n\| \geq \|\n^{(t)}\|  \geq (1+o(1))  \frac{w (\n^{(t)})}{q^*}  \geq
	(1+o(1)) \frac{w_*(\n)}{k^2 q^* \log \|\n\| } =  \|\n\|^{1+o(1)}.
	\]
	That is, we get $\|\n^{(t)}\| = \|\n\|^{1+o(1)}$.}  Applying Lemma \ref{L:upper2} with any $0<\eps'<\eps$, we derive that, 
\begin{align*}
	\chi(\G^{(t)})  \leq  \left(1 +\eps' \right) \frac{w(\n^{(t)})}{2 \log (q^*\|\n^{(t)}\|)}
	+    { O\left(   \frac{k q^*\|\n^{(t)}\|\hat q(\n^{(t)})  }{\log^2 (q^*\|\n^{(t)}\|)}\right),}
\end{align*}
with probability   at least {$1 -  \exp\left(-\|\n\|^{ 2-4\sigma +o(1)}\right)$. 
	Using  the union bound for all such event over $t\notin T_{\rm small}$,  
we get that, with sufficiently high probability,
\begin{equation}\label{chi-big}
	\begin{aligned}
		\sum_{t \in [k] \setminus T_{\text{small}}} \chi(\G^{(t)})  &\leq   
		\left(1 +\eps' +o(1) \right) \frac{w_*(\n) + k^2 q^*}{ 2\log (q^*\|\n\|)}  
		+  {O\left(    \frac{k q^* \|\n\|\hat q(\n) }{\log^2 (q^*\|\n \|)}\right)}
		\\
		&\leq  \left(1 +\eps' +o(1) \right) \frac{w_*(\n)}{2 \log (q^*\|\n\|)}.
	\end{aligned}
\end{equation}
{For the first inequality in \eqref{chi-big}, we estimated  the $O(\cdot)$ term using
	\[
	\sum_{t \in [k]\setminus T_{\text{small}}}
	\|\n^{(t)}\|\hat q(\n^{(t)})  =  \sum_{t \in [k]\setminus T_{\text{small}}} \sum_{i\in [k]} q_{ii} n_i^{(t)}
	\leq  \sum_{i\in [k]} q_{ii} n_i = \|\n\|\hat q(\n).
	\]}

Finally, substituting the bounds of  \eqref{chi-small} and \eqref{chi-big} into \eqref{chi-inequality}  and  bounding  
\[ 
1 +\eps' +o(1) \leq 1 +\eps,
\] we complete the proof  of Theorem \ref{T:upper}.

\section{Proof of Theorem \ref{l:Qmatrix}}\label{S:Qmatrix}

%

The part (a) follows directly by the definition.

The part (b) is trivial if $w(\xvec)=0$ as we can take $\yvec = \boldsymbol{0}$. Thus, we can assume that $w(\xvec)>0$.
Observe that the function $\y \mapsto \frac{\yvec^T Q  \yvec}{\|\yvec\|}$ (defined to be $0$ at origin) is  
continuous on the compact set  $K_{\xvec}$ defined by
\[
K_{\xvec}:= \{\yvec \in \pReals^k \st \yvec \preceq \xvec\}.
\] Therefore, there exists a maximiser $\yvec^*=(y_1^*,\ldots, y_k^*) \in K_{\xvec}$. 
Furthermore, for each $i\in [k]$, 
   \[
   f_i(y_i):=    \frac{\yvec^T\, Q\,  \yvec}{\|\yvec\|}  
   = q_{ii} (y_i -a) + \frac{b}{y_1 + \ldots + y_n},
   \]
   where 
   $\yvec$ differ from $\yvec^*$ in the $i$-th component only and   $a = a(\yvec^*)$, $b = b(\yvec^*)$. 
   If $b\leq 0$ then $q_{ii}\neq 0$ because $f_i(y_i^*) = w(\xvec)>0$. This implies that $q_{ii}>0$
   so the  
   function  $y_i\mapsto f_i(y_i)$ is strictly increasing. If $b>0$ then 
   the  function $y_i\mapsto f_i(y_i)$  is strictly convex.  In any case,  the maximum lies on the boundary $\{0, x_i\}$. 
    Repeating the same argument for the other components, we get (b).

 Before proceeding, we introduce an additional notation.
 For a positive integer $\ell$ and $\xvec \in \pReals^k$, let
\begin{align}\label{def:rel}
	w_\ell(\x) := &\min_{(\xvec^{(t)})_{t\in [\ell]}} \ \  \sum_{t=1}^\ell w_Q(\xvec^{(t)})  
	\\ &\text{ subject to } \label{def:rel}  \sum_{t=1}^{\ell} \x^{(t)} = \x  \text{ and  }  \x^{(t)}  \in \pReals^k 
 \text{ for all }
  t\in [\ell].\nonumber
 \end{align}
A minimum system of $\ell$ vectors $(\xvec^{(t)})_{t\in [\ell]}$  in \eqref{def:rel} exists since  $\sum_{t=1}^\ell w_Q(\xvec^{(t)})$  is a continuous function on the compact set  of 
all  systems  that   $\sum_{t=1}^{\ell} \x^{(t)} = \x$  and $\x^{(t)}  \in \pReals^k $.
Using definitions \eqref{w_def} and \eqref{def:mu-star}, we find that 
\[
	w(\xvec) = w_1(\xvec) \geq w_{\ell}(\xvec) \geq w_*(\xvec) \quad \text{ and }\quad w_*(\xvec) = \lim_{\ell \rightarrow \infty } w_{\ell}(\xvec).
\]
We will also use the following identity.
 \begin{lemma}\label{l:identity}	For any $\yvec,\zvec \in \pReals^k$,  we have
         	\begin{align*}
         			 \frac{\yvec^T \, Q\,  \yvec}{\|\yvec\|} + 
              \frac{\zvec^T \, Q\,  \zvec}{\|\zvec\|}
              &-        \frac{(\yvec+\zvec)^T \, Q\,  (\yvec+\zvec)}{\|\yvec\|+\|\zvec\|}
              \\ &= \frac{\|\yvec\|\|\zvec\| }{\|\yvec\|+\|\zvec\|} 
              \left(\frac{\yvec}{\|\yvec\|} - \frac{\zvec}{\|\zvec\|}\right)^T \, Q\, 
              \left(\frac{\yvec}{\|\yvec\|} - \frac{\zvec}{\|\zvec\|}\right),
         	\end{align*}
         	where $\frac{\xvec^TQ \xvec}{\|\xvec\|}$
         	is $0$ if $\|\xvec\|=0$ and 	the RHS is taken to be $0$ if    $\|\yvec\|=0$ or $\|\zvec\| = 0$.
\end{lemma}
\begin{proof}
This follows by expanding 
   $\left(\frac{\yvec}{\|\yvec\|} - \frac{\zvec}{\|\zvec\|}\right)^T \, Q\, 
              \left(\frac{\yvec}{\|\yvec\|} - \frac{\zvec}{\|\zvec\|}\right)$ and 
              $(\yvec+\zvec)^T \, Q\,  (\yvec+\zvec)$,  using linearity,     and direct substitution.
\end{proof}

We proceed to  part (c). Take any $\yvec \in K_{\xvec}$   such that $\frac{\yvec^TQ \yvec}{\|\yvec\|}  = w(\xvec)$, which exists by part (b). 
 Using Lemma \ref{l:identity}, we find that
  \[
 		w_{\ell}(\xvec)
 		\geq  w_{\ell}(\yvec) 
 		= \sum_{t=1}^\ell w(\yvec^{(t)})  
 		\geq  \sum_{t=1}^\ell \frac{(\yvec^{t})^T Q \yvec^{t}}{\|\yvec^{t}\|} \geq  \frac{\yvec^TQ \yvec}{\|\yvec\|} = w(\xvec).
 \] 
 Taking the limit $\ell \rightarrow \infty$, we prove (c).

 The upper bound 
 \[
 	w_*(\xvec)\leq \hat{q}(\xvec)\|\xvec\|  = \sum_{i\in [k]} x_i q_{ii}
 \] 
 follows by definition \eqref{def:mu-star}  taking the system  of $k$ vectors $(\xvec^{(t)})_{t\in[k]}$, where, for each $t \in [k]$,  the $t$-th component
 of  $\xvec^{(t)}$ equals $x_t$ while other components are 0.
  Also we have  $\hat{q}(\xvec)\|\xvec\| \leq q^* \|\xvec\|$. 
 Next we prove the lower bound for $w_*(\xvec)$ of part (d).   Let  $(\xvec^{(t)})_{t\in \ell }$ be such that 
 $w_{\ell}(\xvec) = \sum_{t\in [\ell]} w(\xvec^{(t)})$ and  $\sum_{t\in [\ell]} \xvec^{(\ell)}= \xvec$.
By the Cauchy-Schwarz inequality, we find that 
 \[
 	w(\xvec^{(t)}) \geq \frac{(\xvec^{(t)})^T Q \xvec^{(t)}}{ \|\xvec^{(t)}\|} 
 	\geq \frac{\sum_{i\in [k]}q_{ii} (x_i^{(t)})^2}{\|\xvec^{(t)}\|}
 	\geq  \frac{\left(\sum_{i\in [k]}q_{ii} x_i^{(t)}\right)^2}{\|\xvec^{(t)}\| \sum_{i\in [k]}q_{ii}}.
 \]
Using  the Cauchy-Schwarz inequality again, we obtain
\[
	\sum_{t\in [\ell]} \|\xvec^{(t)}\| \cdot
	\sum_{t\in [\ell]} 
	\frac{\left(\sum_{i\in [k]}q_{ii} x_i^{(t)}\right)^2}{\|\xvec^{(t)}\| }
	\geq  \left( \sum_{t\in [\ell]} 
	\sum_{i\in [k]}q_{ii} x_i^{(t)}\right)^2 = \left(\sum_{i\in [k]}q_{ii} x_i\right)^2 =
	(\hat{q}(\xvec)^2) \|\xvec\|^2.
\]
Therefore,
\[
		w_{\ell} (\xvec) = \sum_{t\in [\ell]} \xvec^{(\ell)} \geq 
		 \frac{(\hat{q}(\xvec)^2)}{\sum_{i \in [k]}q_{ii} } \|\xvec\|.
\]
 Taking the limit $\ell \rightarrow \infty$ and  observing  $\sum_{i \in [k]}q_{ii} \leq kq^*$, we complete the proof of (d).

 For (e), consider  any two vector systems $\calS \in \calF(\xvec)$ and  $\calS' \in \calF(\xvec')$; see definition \eqref{def:mu-star}. 
 Then,   the union system $\calS \cup \calS'$ belongs to  $\calF(\xvec+\xvec')$. Thus,
 \[
 	\sum_{\yvec  \in \calS} w(\yvec) + 	\sum_{\yvec  \in \calS'} w(\yvec) 
 	=\sum_{\yvec  \in \calS\cup \calS'} w(\yvec) \geq w_*(\xvec+\xvec'). 	 
 \]
 Taking the infimum over $\calS, \calS'$, we  get (e).

 For (f), assume $\ell>k$. Then, we can find real constants $c^{(t)}$, $t=1,\ldots,\ell$, such that 
   \[
      \sum_{t=1}^\ell c^{(t)} \xvec^{(t)} =0.
   \]
   Next, we show that  $\xvec^{(t)}_{\varepsilon}=(1 + \varepsilon c^{(t)}) \xvec^{(t)}$ gives another optimal solution of \eqref{def:rel}. 
   Observe that $ \sum_{t=1}^\ell \xvec^{(t)}_{\varepsilon} = \xvec$.  If 
   $|\varepsilon|$ is sufficiently  small that  $\xvec^{(t)}_{\varepsilon} \in \pReals^k$
   then    
   \[
      f(\varepsilon):=\sum_{t=1}^\ell w(\xvec^{(t)}_{\varepsilon}) = \sum_{t=1}^\ell   (1 + \varepsilon c^{(t)}) w(\xvec^{(t)}),
    \] 
   that is, $f(\varepsilon)$ is a linear function of $\varepsilon$. Since $\varepsilon = 0$ gives the minimum value of  $f(\varepsilon)$,  it should be a constant function.  Then, we can make at least one of 
   $\xvec^{(t)}_{\varepsilon}$ to be trivial while others remain in $\pReals^k$ 
   without changing the value of the target function   $\sum_{t=1}^{\ell} w(\xvec^{(t)}_{\varepsilon}) $.
   This implies 	 $w_\ell(\xvec) = w_{\ell-1}(\xvec)$. 
   Repeating these arguments several times we find that 
    \[
    	    w_\ell(\xvec) = w_{\ell-1}(\xvec)  = \cdots= w_k(\xvec).
    \]
  Taking the limit $\ell \rightarrow \infty$, we get (f).
   
Finally, we proceed to part (g). 
Using part (f), we can find a  system  $(\y^{(t)})_{t\in [k]}$  such that
       \[
	    \sum_{t=1}^k w(\yvec^{(t)}) = w_*(\xvec) \qquad \text{ and } \qquad \sum_{t=1}^k \yvec^{(t)} = \xvec.
       \]
       In particular, by definition of $w_*(\cdot)$, we find  that
       \begin{equation}\label{eqwww}
       	w_*(\yvec^{(t)})  = w(\yvec^{(t)}).
       \end{equation}
         Define $\xvec^{(t)} := \lfloor\yvec^{(t)} \rfloor$. 
        Combining \eqref{eqwww} and  parts (d), (e), we also find  that for all $t \in [k]$
\[
  w(\yvec^{(t)}) = w_{*}(\yvec^{t})  \leq w_*(\xvec^{(t)}) + w_*(\yvec^{(t)} -  \xvec^{(t)})
  \leq  w_*(\xvec^{(t)}) + q^* \|\yvec^{(t)} -  \xvec^{(t)}\| \leq w_*(\xvec^{(t)}) + kq^*.
\]
Thus,  we get that 
\[
	 \sum_{t=1}^k w(\xvec^{(t)}) 
	 \geq \sum_{t=1}^k (w(\yvec^{(t)}) - kq^*) = w_*(\xvec) - k^2q^*. 
\]
Now, we can increase some components of $\xvec^{(t)}$ to ensure that $\sum_{t \in [k]} \xvec^{(t)} = \xvec$.
By part (a), this would only increase the values of $w(\xvec^{(t)}) $. This completes the proof of part (g) 
and Theorem \ref{l:Qmatrix}.


\subsection*{Acknowledgements}
We thank the referee for valuable comments and suggestions that helped us improve the paper.

%
%

%

\end{document}